\documentclass[11pt]{article}
\usepackage[a4paper,margin=26mm]{geometry}
\usepackage[T1]{fontenc}
\usepackage{lmodern}
\usepackage{microtype}
\usepackage{xspace}
\usepackage{amsmath,amssymb,amsthm,mathtools,bm}
\usepackage{booktabs,tabularx,array,multirow}
\usepackage{graphicx}
\usepackage{float}
\usepackage{xcolor}
\usepackage{enumitem}
\usepackage{algorithm}
\usepackage{algpseudocode}
\usepackage{placeins}
\usepackage{tikz}
\usetikzlibrary{arrows.meta,positioning}
\usepackage{pgfplots}
\pgfplotsset{compat=1.18}
\usepackage[numbers,sort&compress]{natbib}
\usepackage[colorlinks=true,linkcolor=blue!45!black,citecolor=blue!45!black,urlcolor=blue!45!black,
  pdftitle={TOGEARI: Interaction-Space Preconditioning for Condensed Finite-Element Systems with IPC Contact},
  pdfauthor={Yanlin Liu, Chao Huang, Kaixiang Yao, Yao Shen},
  pdfsubject={Interaction-space preconditioning for condensed finite-element contact systems},
  pdfkeywords={finite elements, IPC contact, preconditioning, FGMRES, nearly incompressible elasticity, low-rank update}]{hyperref}
\usepackage[nameinlink,capitalise]{cleveref}

\newcommand{\R}{\mathbb{R}}

\newcommand{\rank}{\operatorname{rank}}
\newcommand{\spec}{\operatorname{spec}}
\newcommand{\diag}{\operatorname{diag}}
\newcommand{\TOGEARI}{\textsc{Togeari}\xspace}
\newcommand{\TOGEARIM}{\textsc{Togeari-M}\xspace}
\newcommand{\TOGEARIR}{\textsc{Togeari-R}\xspace}
\newcommand{\norm}[1]{\left\lVert #1\right\rVert}

\newtheorem{theorem}{Theorem}
\newtheorem{proposition}[theorem]{Proposition}
\newtheorem{lemma}[theorem]{Lemma}
\newtheorem{corollary}[theorem]{Corollary}
\theoremstyle{definition}

\theoremstyle{remark}
\newtheorem{remark}[theorem]{Remark}

\AddToHook{env/theorem/begin}{\crefalias{theorem}{theorem}}
\AddToHook{env/proposition/begin}{\crefalias{theorem}{proposition}}
\AddToHook{env/lemma/begin}{\crefalias{theorem}{lemma}}
\AddToHook{env/corollary/begin}{\crefalias{theorem}{corollary}}
\AddToHook{env/definition/begin}{\crefalias{theorem}{definition}}
\AddToHook{env/remark/begin}{\crefalias{theorem}{remark}}

\AddToHook{cmd/appendix/after}{%
  \crefalias{section}{appendix}%
  \crefalias{subsection}{subappendix}%
  \crefalias{subsubsection}{subsubappendix}}

\title{\textbf{TOGEARI: Interaction-Space Preconditioning\\for Condensed Finite-Element Systems with IPC Contact}}
\author{Yanlin Liu$^{1,\dagger}$ \and Chao Huang$^{2,\dagger}$ \and
Kaixiang Yao$^{3,\dagger}$ \and Yao Shen$^{2,*}$}
\date{\small $^1$The University of Melbourne, Melbourne, Australia\\
$^2$Shanghai Jiao Tong University, Shanghai, China\\
$^3$Southern University of Science and Technology, Shenzhen, China\\
$^\dagger$Yanlin Liu, Chao Huang, and Kaixiang Yao share first authorship.\\
$^*$Corresponding author: \texttt{yaoshen@sjtu.edu.cn}}

\begin{document}
\maketitle

\begin{abstract}
TOGEARI constructs a compressed contact correction for condensed finite-element
equations with incremental potential contact (IPC).
A core-conditioned interaction matrix ranks contact combinations by their
mechanical responses; a raw-factor proxy acquires only the retained responses.
For a symmetric positive-definite core, the reference analysis gives the exact spectrum,
scaled inverse error and optimal worst omitted interaction.
The selected responses are maintained through signed contact changes using
supported operator images and a small coarse matrix.
A complementary Schur spectrum and perturbation bounds connect this
construction to the full current linear equation, damped local Newton steps
and compatible episode reuse.
On a 325,260-coordinate finger/cup problem, eight selected directions
reduce Krylov work from eleven columns to five, matching full contact.
Two reversed-order load studies give matching anchored displacements and
accepted normal reactions, with 28.8--29.3 percent less warm time and
5.1--5.3 percent less complete time than per-step direct refactorization.
A completed twelve-episode run saves 18.8 percent of total time
against direct and 2.8 percent against a held factor.
Full response caching has comparable time with 7.5 times the response storage.
A compressed repeat terminates at an energy line search; separate diagnostics
identify constitutive cancellation and terminal-backtracking sensitivity.
Larger-contact, held-out-load and frictional comparisons show where
compression or preparation loses its advantage.
The complete Newton equation remains the residual operator throughout;
the convergence guarantees retain their stated reference and local hypotheses.
\end{abstract}


\section{Introduction}
\label{sec:intro}

Localized contact stiffness induces mechanical responses throughout a
finite-element body.
A useful contact correction must therefore identify combinations of
interactions through the responses they induce in the mechanical system.
TOGEARI constructs such a compressed correction and maintains it as
contact changes. The question is how much useful response can be retained,
with its acquisition, application and update costs counted together.

Let $n$ be the number of free displacement coordinates and $m$ the number
of contact-factor rows. After local pressure elimination and boundary
reduction, write the complete Newton matrix as
\begin{equation}
 A=M+U^TU+N.
\label{eq:main-class}
\end{equation}
Here $M\in\R^{n\times n}$ is the reusable mechanical core,
$U\in\R^{m\times n}$ is a row factor of the admitted
positive-semidefinite contact tangent, and $N$ contains the remaining
pressure-geometric, follower-load and reconstruction terms.
The correction approximates contact within this decomposition.
For a right-hand side (RHS) $b$, Krylov iteration computes a correction $x$
for the complete equation $Ax=b$ and checks its residual $b-Ax$.

Truncated Operator-Gram Eigenspace Approximation of Relevant Interactions
(\TOGEARI) chooses combinations of the rows of $U$.
A contact load is first mapped through the core inverse and then returned
to contact coordinates. Its interaction matrix is
\begin{equation}
 G_M=UM^{-1}U^T.
\label{eq:metric-gram-intro}
\end{equation}
For a symmetric positive-definite (SPD) core, the leading eigenspace
minimizes the largest omitted core-normalized contact energy.
This gives the mechanically conditioned selector, \TOGEARIM.
The inexpensive realization, \TOGEARIR, selects from
\begin{equation}
 G_0=UU^T
\label{eq:raw-gram-intro}
\end{equation}
before acquiring only the retained mechanical responses.
The metric construction explains relevance; their matched comparison
determines when the raw proxy retains the same useful content.

The construction has three connected parts.
The locally condensed finite-element/contact split identifies the mechanical core and
contact factor in the implemented displacement equation.
Reference spectral analysis then characterizes the omitted interactions
and the inverse-action error of a chosen space.
Finally, a supported balanced action maintains the selected responses
through signed contact changes around a fixed material or complete-matrix
reference. Its complementary spectrum and perturbation analysis connect
the reference to the complete equation, a damped Newton step and
uniform reuse over compatible episode families.
A response is the displacement produced by a selected load; its current
image is the current reference operator applied to that response.
Updating those images avoids reacquiring the full response vectors.

\begin{figure}[htbp]
\centering
\begin{tikzpicture}[
  box/.style={draw=blue!45!black,rounded corners=2pt,
    fill=blue!3,text width=3.8cm,minimum height=1.45cm,
    align=center,font=\small,inner sep=5pt},
  flow/.style={-{Latex[length=2mm]},thick,draw=blue!55!black}]
\node[box] (mechanics) at (0,0)
  {Mechanical split\\Core and contact factor\\Full Newton operator};
\node[box] (select) at (4.7,0)
  {Select interactions\\Core-response metric\\or raw-factor proxy};
\node[box] (responses) at (9.4,0)
  {Acquire responses\\Material or complete\\reference factor};
\node[box] (update) at (9.4,-2.5)
  {Update current images\\Signed contact changes\\Small coarse matrix};
\node[box] (solve) at (4.7,-2.5)
  {Solve the full equation\\Balanced inverse action\\Full residual and trial};
\node[box] (reset) at (0,-2.5)
  {Compatible reset\\Restore physical history\\Retain numerical data};
\draw[flow] (mechanics) -- (select);
\draw[flow] (select) -- (responses);
\draw[flow] (responses) -- (update);
\draw[flow] (update) -- (solve);
\draw[flow] (solve) -- (reset);
\draw[flow] (solve.south) -- (4.7,-4.0)
  -- node[below,font=\scriptsize]{next equation} (9.4,-4.0) -- (update.south);
\draw[flow] (reset.south) -- (0,-4.65) -- (11.7,-4.65)
  -- (11.7,-2.5) -- (update.east);
\end{tikzpicture}
\caption{The complete interaction-space construction. Mechanical and contact
data determine the selected space; retained responses are acquired for the
chosen reference. Signed contact changes update their supported images and
coarse matrix. Right-preconditioned flexible generalized minimal residual
(FGMRES) iteration \citep{saad1993fgmres} uses the full current equation. A compatible episode
reset restores physical history while retaining eligible numerical
preparation. The equations and algorithm below specify each step.}
\label{fig:maintained-lifecycle}
\end{figure}
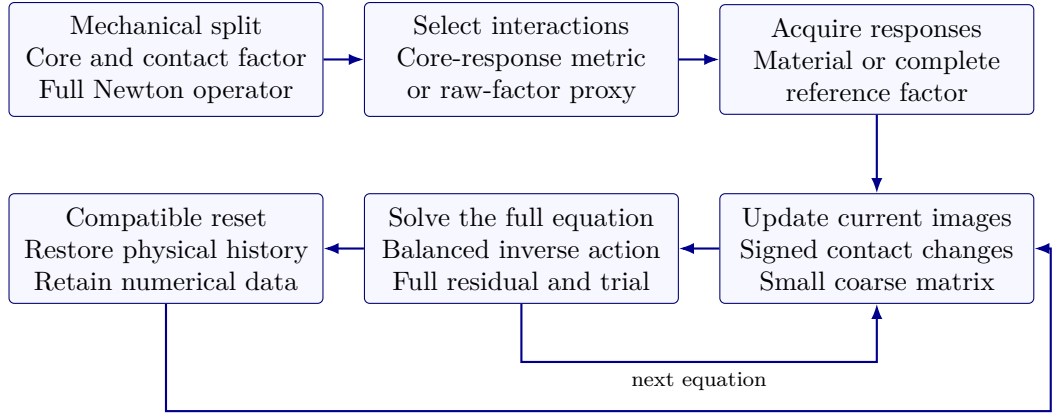
\FloatBarrier

The mathematical ingredients include established rank-structured spectral
analysis, Woodbury inversion, balanced two-level preconditioning and
inexact-Newton theory
\citep{helmberg2024preconditioned,bock2024bregman,tang2009twolevel,
dembo1982inexact,eisenstat1996forcing}.
The contribution is their mechanically defined interaction-space
construction and its linked truncation, maintenance and complete-equation
consequences. \Cref{sec:related} identifies the closest constructions and
their roles. The guarantees use the stated SPD, perturbation and local
regularity hypotheses;
the numerical comparisons retain their independently checked complete
residuals and their actual factorization scope.

Three comparisons establish the computational meaning.
Against the same core, selected contact information reduces the primary
equation's work from eleven Krylov columns to five; dimension-matched
row-norm and random choices require twelve.
Eight raw or metric directions match the five columns of full contact.
Maintaining their current images then preserves useful work across changing
equations, with substantially cheaper updates than reacquiring the space.
In two reversed-order quasi-static studies, complete anchoring saves
28.8--29.3\% of warm execution time against per-step direct refactorization
and 5.1--5.3\% of complete run time, with tightly matching anchored
displacements and accepted normal reactions.
Held-factor and full-cache comparisons identify the smaller incremental
benefit of the correction and its response-storage advantage.

The larger reset program retains one factor and eight responses.
Its completed compressed run saves 18.8\% of total direct time and
2.8\% against a held factor; full caching has comparable time with
7.5 times the response storage.
An exceptional compressed repeat limits nonlinear robustness.
Larger-contact and held-out-load cases also show that a small space,
a better worst-mode quantity or a better first direction can fail to
improve the complete calculation.
These observations delimit the construction without changing its equations.

The paper develops the mechanical split, selection and maintained action
in \cref{sec:algebra,sec:method,sec:selector-application}, and proves the
reference and current-equation results in
\cref{sec:theory,sec:convergence-reuse}.
The numerical argument then follows useful interaction content,
maintenance and complete-work economics, with the adverse regimes and
supporting controls attached to the conclusions they qualify.

\section{Related work and point of comparison}
\label{sec:related}

IPC combines barrier contact and continuous collision detection
\citep{li2020ipc}. GIPC's local barrier eigensystems, contact-aware multilevel
Schwarz and StiffGIPC address the resulting stiff systems
\citep{huang2024gipc,wu2022gpu,huang2025stiffgipc}.
Barrier-augmented Lagrangians change the nonlinear formulation
\citep{guo2024bal}; AGIPC changes the active algebraic system
\citep{wang2026agipc}.
These are broader solver constructions. Our comparisons isolate a contact
correction on the same Newton equation and then measure its effect in
complete nonlinear histories.

Subspace-preconditioned projective dynamics also addresses cloth contact
\citep{li2023subspace}.
Contact-compliance and subspace methods provide the closest setting.
Preconditioner-based compliance, low-frequency factorization and GPU
Schur evaluation already support large contact/friction calculations
\citep{zeng2022contact}. Supported triangular responses and compliance can
also be reused between numerical refactorizations
\citep[Chapter~4]{zeng2023thesis}.
MAS-PNCG selects locally important Sparse-Input Woodbury updates
\citep{zhang2026maspncg}; contact-space and Delassus constructions appear
in differentiable contact solvers \citep{zeng2026gradients}.
Contact-interface AMG corrections, adaptive spectral coarse spaces and
material-aware reduced bases offer related approaches
\citep{petrides2025amgf,efendiev2012coarse,agullo2019spectral,trusty2025bases}.
TOGEARI asks which prescribed-rank combinations of an available contact
factor should be retained relative to the mechanical core.

The generic spectral and inverse tools are established
\citep{stewart1990matrix,saad2003iterative}.
Helmberg's rank-structured SPD theory applies with
$D_{\mathrm{Helmberg}}=M$ and $V_{\mathrm{Helmberg}}=U^T$
\citep{helmberg2024preconditioned}.
Bock and Andersen study unscaled and core-scaled low-rank approximations,
including spectral bounds, optimality and randomized constructions
\citep{bock2024bregman}; their later truncation treats signed factorization
errors \citep{bock2025truncation}.
The present reference results specialize these foundations to the
condensed contact factor, expose its mechanical omitted-energy meaning,
and connect the selected space to the maintained complete-equation action.
The raw proxy's usefulness is determined by matched observations.

Maintenance also has direct antecedents. Changing contact directions can
act through a persistent geometric mapping and compliance
\citep{zeng2023implicit}. Balanced two-level preconditioning updates an
operator on a retained coarse space \citep{tang2009twolevel}.
Delayed factorization and Krylov recycling have been combined for
contact/friction continuation \citep{kuether2024recycling}.
Prepared Schur eigenspaces can be restricted to changing inactive sets
\citep{kadeethum2026online}.
Our realization maintains compressed responses through supported signed
contact changes inside the full FEM equation. Held-factor and matched
full-cache comparisons locate the contribution of compression.
The experiments do not rank it against a general recycling Krylov method.

The material decomposition uses classical local pressure condensation
\citep{simo1985volume}. Parameter-robust elasticity preconditioners and
quadratic or matrix-free multigrid provide alternative core solvers
\citep{klawonn1998penalty,wieners2000multigrid,yang2015quadratic,
davydov2020matrixfree,brown2022performance,schussnig2025matrixfree}.
The reported timings use direct core factors; \cref{app:core-regimes}
states the additional conditions for an approximate core.
The current-equation and nonlinear analysis draws on balanced spectra,
GMRES polynomial estimates and inexact Newton theory
\citep{tang2009twolevel,liesen2020ideal,dembo1982inexact,eisenstat1996forcing}.
Established spectral normalization and residual-minimizing combinations
supply the separate strength controls
\citep{frangella2023nystrom,greif2011multipreconditioning,
greif2013weights,ayuso2014combined}.

\section{The condensed mechanical/contact split}
\label{sec:algebra}

For every positive integer $k$, $I_k$ denotes the $k\times k$ identity matrix.
The symbol $\norm{\cdot}_2$ denotes the Euclidean norm for vectors and the
induced spectral norm for matrices. The Frobenius norm $\norm{\cdot}_F$
is the square root of the sum of squared matrix entries.
For symmetric matrices $X$ and $Y$,
$X\preceq Y$ denotes the Loewner order, meaning that $Y-X$ is positive
semidefinite; $X\succeq Y$ denotes the reverse relation.

The construction starts from the displacement equation actually solved.
Tet10/P0 denotes a ten-node quadratic tetrahedron with one constant cell
pressure, eliminated locally. Let $u_e$ be its displacement vector,
$c_e(u_e)$ its integrated volume constraint, $V_e>0$ its reference volume,
$\kappa_e>0$ its bulk modulus and $\Psi_{\mathrm{iso},e}$ its isochoric energy.
The local pressure compliance is $D_e=V_e/\kappa_e$.
The latent cell pressure $p_e$ is eliminated before the global solve:
\begin{equation}
 R_{p,e}=c_e-D_ep_e=0,\qquad
 p_e=D_e^{-1}c_e,\qquad B_e=\frac{\partial c_e}{\partial u_e}.
\label{eq:local-pressure}
\end{equation}
Substitution gives the condensed energy and its complete tangent,
\begin{align}
 \Psi_e(u_e)&=\Psi_{\mathrm{iso},e}(u_e)
                 +\tfrac12D_e^{-1}c_e(u_e)^2,
 \label{eq:condensed-energy}\\
 K_e^{\mathrm{mat}}
 &=\underbrace{\nabla^2\Psi_{\mathrm{iso},e}}_{K_{\mathrm{iso},e}}
   +\underbrace{B_e^TD_e^{-1}B_e}_{K_{\mathrm{vol},e}\succeq0}
   +\underbrace{p_e\nabla^2c_e}_{G_e}.
 \label{eq:material-split}
\end{align}
The positive condensed outer product and signed pressure-geometric
curvature are different parts of this derivative.
Let $B,D$ assemble the local $B_e,D_e$ and let $R_p$ be the assembled
pressure residual. Equivalently, elimination of the pressure increment
from a mixed Newton block with off-diagonal blocks $B^T,B$ and pressure block $-D$ adds
$B^TD^{-1}B$ to the displacement tangent and $-B^TD^{-1}R_p$ to its RHS.
Here the pressure residual is already zero after substitution.

All subsequent matrices use the $n$ free displacement coordinates.
Let $q$ be the number of condensed cells, $B\in\R^{q\times n}$ the assembled
constraint Jacobian and $D=\operatorname{diag}(D_e)$.
The reusable mechanical core for the primary construction is
\begin{equation}
 M=K_0+K_{\mathrm{vol}},\qquad
 K_{\mathrm{vol}}=B^TD^{-1}B,\qquad
 K_0=K_{\mathrm{plain}}+K_{\mathrm{iso}}+K_{\mathrm{num}}.
\label{eq:experimental-factor-base}
\end{equation}
Here $K_{\mathrm{plain}}$ is the plain-displacement material tangent,
$K_{\mathrm{iso}}$ the condensed cells' isochoric contribution and
$K_{\mathrm{num}}$ any specified numerical regularization.
The signed assembled pressure curvature $G$, follower-pressure tangent
$F_{\mathrm{pneu}}$ and other specified terms remain in the complete equation.
The primary and near-repeat states have zero friction, mass, damping,
$K_{\mathrm{num}}$ and other additional terms.
The larger comparisons explicitly state their different complete
noncontact core $\overline M$.

\subsection{A contact factor in free displacement coordinates}

Let $c$ be the number of active point--triangle or edge--edge pairs.
For pair $a\in\{1,\ldots,c\}$, let
$\widehat K_a^+\in\R^{12\times12}$ be the positive-semidefinite (PSD) projection of its
normal-barrier Hessian.
The map $T_a\in\R^{12\times d_a}$ takes its pair-local quadratic
finite-element coordinates to the four collision vertices.
The reported discretization has $d_a=36$.
Define
\begin{equation}
 C_{a,\mathrm{asm}}=T_a^T\widehat K_a^+T_a,\qquad
 \tfrac12(C_{a,\mathrm{asm}}+C_{a,\mathrm{asm}}^T)=V_a\Lambda_aV_a^T.
\label{eq:local-contact-pullback}
\end{equation}
Here $V_a$ is orthogonal and $\Lambda_a$ is diagonal.
Retain the eigenvalues above the specified numerical rank tolerance,
with corresponding factors $V_{a,+},\Lambda_{a,+}$ and count $k_a$.
The local row factor is
\begin{equation}
 L_a=\Lambda_{a,+}^{1/2}V_{a,+}^T,\qquad
 C_{a,\mathrm{fac}}=L_a^TL_a.
\label{eq:local-contact-factor}
\end{equation}

Let $n_{\mathrm{full}}$ be the complete-coordinate dimension and
$E\in\R^{n_{\mathrm{full}}\times n}$ inject free coordinates,
so $E^TE=I_n$ and the free-coordinate projector is $P_f=EE^T$.
Let $S_a\in\R^{d_a\times n_{\mathrm{full}}}$ restrict a complete displacement vector to pair $a$.
Stacking the local rows gives
\begin{equation}
 U=\begin{bmatrix}L_1S_1E\\ \vdots\\L_cS_cE\end{bmatrix},
 \qquad m=\sum_{a=1}^c k_a,\qquad
 U^TU=\sum_{a=1}^c E^TS_a^TC_{a,\mathrm{fac}}S_aE.
\label{eq:global-contact-factor}
\end{equation}
Thus each row is a contact interaction in the original free coordinates.
The factor-reconstruction remainder is retained explicitly:
\begin{equation}
 C_{\mathrm{asm}}=\sum_a E^TS_a^TC_{a,\mathrm{asm}}S_aE,\qquad
 E_{\mathrm{ipc}}=C_{\mathrm{asm}}-U^TU.
\label{eq:free-contact-reconstruction}
\end{equation}
The specified pair-local truncation therefore changes the reference factor,
with its discrepancy still present in the complete operator.

For zero prescribed displacement increments, let $A_{\mathrm{full}}$ and
$R_{\mathrm{full}}$ be the boundary-conditioned matrix and complete-coordinate
residual. With $K_{\mathrm{other}}$ collecting other specified contributions,
restriction gives
\begin{equation}
 A=E^TA_{\mathrm{full}}E=M+U^TU+N,\qquad
 N=G+F_{\mathrm{pneu}}+K_{\mathrm{other}}+E_{\mathrm{ipc}},
 \qquad b=-E^TR_{\mathrm{full}}.
\label{eq:full-decomposition}
\end{equation}
Indeed, $P_fE=E$ and the constrained diagonal vanishes under $E^T(\cdot)E$.
This also shows the equivalence to assembling each complete-coordinate
category as $P_fXP_f$ before inserting the Dirichlet diagonal.
Nonzero prescribed increments would require a lift from the unprojected
physical operator; they are outside the reported systems.
Define the contact/core reference by
\begin{equation}
 H=M+U^TU,\qquad A=H+N.
\label{eq:general-A}
\end{equation}

The primary $36$ pairs comprise $10$ point--triangle and $26$ edge--edge
pairs and give $94$ stored rows.
Their free-category reconstruction discrepancies are
$5.63\times10^{-16}$ and $4.81\times10^{-13}$ in relative Frobenius norm.
Small negative eigenvalues of nominally PSD pullbacks were below the specified
relative PSD tolerance; every retained mode satisfied that tolerance.
Those observations qualify the numerical factorization, with its nonzero
remainder retained in $N$.

A fixed nonsingular symmetric $M$ permits core-conditioned selection;
$M\succ0$ supplies its energy and spectral interpretation.
The reported primary large systems use a fixed cuDSS general-matrix factor.
Global Cholesky and inertia were unrecorded, so primary global-core SPD
remains unresolved. The follower tangent is symmetric to float64 roundoff on
those states, but the full class permits nonsymmetric $A$.
The inverse construction below states its weaker algebraic requirements;
the experiments check the complete residual.

\section{Interaction-space preconditioner}
\label{sec:method}

The abstract problem is to replace the contact update $U^TU$ in
$H=M+U^TU$ by a rank-$r$ factor-space projection while leaving the reusable
core $M$ intact.  Three layers of the construction have different hypotheses.
Given an orthonormal factor-space basis, the Woodbury action requires a
nonsingular core and a nonsingular reduced system; symmetry is unnecessary.
Selection from the core-response matrix uses a fixed symmetric nonsingular
core, which may be indefinite.  The energy ordering, exact spectrum, and
optimality results specialize further to an SPD core.  This section defines the selectors and retained spaces. Their inverse
application and maintained realization follow in \cref{sec:selector-application};
\cref{sec:theory} then gives the reference spectral analysis.

\subsection{Core-conditioned relevance}
Assume first that $M=M^T$ is fixed and nonsingular.  This is the operational
hypothesis for the core-conditioned eigenselector.  It permits an indefinite
core.

For an interaction coefficient $\gamma\in\R^m$, define the displacement load
$f_\gamma=U^T\gamma$ and the core response $z_\gamma=M^{-1}f_\gamma$.  The
interaction-response matrix is $G_M=UM^{-1}U^T$.  Mapping the response back to
interaction coordinates gives
\begin{equation}
 \gamma\ \xrightarrow{\ U^T\ }\ f_\gamma
 \ \xrightarrow{\ M^{-1}\ }\ z_\gamma
 \ \xrightarrow{\ U\ }\ G_M\gamma,
 \qquad G_M=UM^{-1}U^T.
\label{eq:metric-response}
\end{equation}
The symmetry of $M$ makes $G_M$ symmetric. Let $q_i$, $i=1,\ldots,m$,
be its orthonormal eigenvectors and write its eigendecomposition as
\begin{equation}
 G_M=Q\Lambda Q^T,
 \quad Q\in\R^{m\times m},
 \quad \Lambda=\diag(\lambda_1,\ldots,\lambda_m),
 \quad \lambda_1\ge\cdots\ge\lambda_m.
\label{eq:metric-gram}
\end{equation}
For $0\le r\le m$, let $Q_r=[q_1,\ldots,q_r]$ and
$\Lambda_r=\diag(\lambda_1,\ldots,\lambda_r)$.  The retained row factor $R_r$
and enriched preconditioner $P_r$ are
\begin{equation}
 R_r=Q_r^TU,
 \qquad P_r=M+R_r^TR_r.
\label{eq:Pr}
\end{equation}
Since $R_rM^{-1}R_r^T=\Lambda_r$, the inverse action exists when
$I_r+\Lambda_r$ is nonsingular.  An indefinite $M$ therefore leaves the
algebraic correction intact.  In that regime, the prototype uses the descending
algebraic eigenvalue order displayed in \eqref{eq:metric-gram}.  Signed response
values have no positive-energy ordering, so this choice is empirical and is
separate from the minimax solution proved below for an SPD core.

Now suppose that $M$ is SPD.  Then
\begin{equation}
 \gamma^TG_M\gamma=f_\gamma^TM^{-1}f_\gamma=z_\gamma^TMz_\gamma.
\label{eq:metric-work}
\end{equation}
Thus $G_M$ ranks combinations of contact curvature by the core-energy of their
induced displacement responses. This is the precise sense in which an
interaction is relevant to the correction.

Let $s$ denote the rank of the interaction factor in core-scaled coordinates:
\begin{equation}
 s=\rank(UM^{-1/2})=\rank U.
\label{eq:interaction-rank}
\end{equation}
The first $s$ eigenvalues in \eqref{eq:metric-gram} are positive and the
remaining eigenvalues are zero.  For $0\le r\le s$, the retained $Q_r$ is the
leading left singular space of $UM^{-1/2}$.  Set
$Q_s=[q_1,\ldots,q_s]$ and
$\Lambda_s=\diag(\lambda_1,\ldots,\lambda_s)$.
For $i=1,\ldots,s$, define the induced response $x_i$ of $q_i$ and its
core-normalized form $\widehat x_i$ by
\begin{equation}
 x_i=M^{-1}U^Tq_i,
 \qquad \widehat x_i=x_i/\sqrt{\lambda_i}.
\label{eq:normalized-response}
\end{equation}
The normalized response satisfies $\widehat x_i^TM\widehat x_i=1$ and
$\widehat x_i^TU^TU\widehat x_i=\lambda_i$. Large eigenvalues therefore
identify responses with a large contact-to-core Rayleigh quotient.  This is a
tangent-energy interpretation; penetration, force, and nonlinear trajectory
error require separate estimates.

\section{Selector realization and inverse application}
\label{sec:selector-application}

\subsection{Raw-factor proxy}
For \TOGEARIR, write the eigendecomposition of the raw Gram matrix $G_0$ as
\begin{equation}
 G_0=UU^T=\widehat Q\widehat\Lambda\widehat Q^T,
\label{eq:raw-gram}
\end{equation}
with the eigenvalues in $\widehat\Lambda$ ordered nonincreasingly.  Let
$\widehat Q_r$ contain the corresponding $r$ leading eigenvectors.  The raw
factor-space projector and retained row factor are
\begin{equation}
 \widehat\Pi_r=\widehat Q_r\widehat Q_r^T,
 \qquad \widehat R_r=\widehat Q_r^TU.
\label{eq:raw-selector}
\end{equation}
This is the Euclidean rank-$r$ spectral truncation of $U^TU$.  It omits the
core metric during selection.  Under its SPD hypothesis,
Proposition~\ref{prop:selector-optimality} supplies a posterior condition bound once
the worst omitted core-normalized interaction in \eqref{eq:eta} is evaluated
for the raw projector. The raw eigenvalues alone leave that
metric quantity undetermined.

\subsection{Woodbury action}
Let $Q_r\in\R^{m\times r}$ be any supplied orthonormal factor-space basis and
set $R_r=Q_r^TU$.  Suppose that $M$ is nonsingular and define
\begin{equation}
 Z_r=M^{-1}R_r^T,
 \qquad S_r=I_r+R_rZ_r.
\label{eq:ZS}
\end{equation}
If $S_r$ is nonsingular, then for $v\in\R^n$ the Woodbury inverse action is
\begin{equation}
 (M+R_r^TR_r)^{-1}v
 =M^{-1}v-Z_rS_r^{-1}R_rM^{-1}v.
\label{eq:woodbury}
\end{equation}
These are the complete algebraic requirements; $M$ may be indefinite or
nonsymmetric.  When $M=M^T\succ0$, the reduced matrix $S_r$ is automatically
SPD.  For
core-response eigenvectors of a symmetric nonsingular $M$,
$R_rM^{-1}R_r^T=\Lambda_r$, so $S_r=I_r+\Lambda_r$.

The explicit frozen-action algorithm is retained at the end of this section
in \cref{alg:togeari}. The maintained lifecycle is given first in
\cref{alg:maintained-lifecycle}.

The inverse formula also accepts a random or externally supplied basis.
\Cref{alg:togeari} includes the selection that defines TOGEARI.
Direct formation of $G_M$ uses the $m$ responses $M^{-1}U^T$;
the equivalent supported-compliance construction in
\cref{sec:high-contact-qualification} obtains the metric space through
coordinate responses. \TOGEARIR selects from $UU^T$ before acquiring
only the $r$ retained responses.
For a nonsymmetric core, use a raw or supplied basis; symmetric-indefinite
metric selection retains its empirical signed ordering.
The metric guarantees require the stated SPD hypotheses.
Approximate and variable-core regimes are described in \cref{app:core-regimes}.

The GPU implementation used in the experiments symmetrizes the
computed reduced matrix $S_r$ and factors it by Cholesky. Successful reduced
Cholesky is therefore a numerical requirement of that implementation. A
general LU factorization of $M$ alone does not establish this requirement.
The algebraic cases in which $S_r$ is indefinite or nonsymmetric require an
appropriate reduced factorization and separate numerical verification;
they are outside the reported GPU comparisons. The complete Newton
matrix $A$ continues to enter FGMRES without this reduced-matrix restriction.

\subsection{Contact support and the cost of a retained space}
\label{sec:response-representations}

The rank of a correction and the storage of its inverse action are separate
design choices. Contact-compliance methods already exploit the restriction of
contact loads to a small set of mechanical coordinates
\citep{zeng2022contact}. We apply the same Woodbury identity to distinguish
the cost of compression from the cost of storing full displacement responses.

In this subsection, $s$ denotes the number of coordinates touched by $U$,
which is distinct from the rank in \eqref{eq:interaction-rank}. Let
$E_s\in\R^{s\times n}$ restrict a displacement vector to those coordinates.
Thus $E_sE_s^T=I_s$. Define the supported contact matrix and core compliance by
\begin{equation}
 C_s=E_sU^TUE_s^T,\qquad W_s=E_sM^{-1}E_s^T.
\label{eq:supported-contact-compliance}
\end{equation}
The contact tangent is $U^TU=E_s^TC_sE_s$. For the retained factor $R_r$ in
\eqref{eq:Pr}, define $R_{r,s}=R_rE_s^T$, so that
$R_r=R_{r,s}E_s$ and
$S_r=I_r+R_{r,s}W_sR_{r,s}^T$.

The cached response $Z_r=M^{-1}E_s^TR_{r,s}^T$ has $nr$ entries and gives
one core application per preconditioner action. An equivalent action can
omit this cache. For an input $v$, set
\begin{equation}
 q=M^{-1}v,\qquad
 g=S_r^{-1}R_{r,s}E_sq,\qquad
 z=q-M^{-1}E_s^TR_{r,s}^Tg.
\label{eq:implicit-retained-response}
\end{equation}
Substitution of $Z_r$ in \eqref{eq:woodbury} proves that
$z=P_r^{-1}v$ whenever the stated inverses exist. This realization uses two
core applications and stores $sr+r^2$ dense entries for the supported root
and reduced factor, up to pivot arrays and temporary working memory. Its reduced
matrix can be formed by streaming $r$ core responses and restricting each
output before discarding the full vector; a complete $W_s$ is unnecessary.

At the full-contact endpoint, define $P_{\rm full}=M+E_s^TC_sE_s$ and
assume that $I_s+C_sW_s$ is nonsingular. Its
support-space action is
\begin{equation}
 q=M^{-1}v,\qquad
 h=(I_s+C_sW_s)^{-1}C_sE_sq,\qquad
 z=q-M^{-1}E_s^Th.
\label{eq:implicit-full-contact}
\end{equation}
Indeed, multiplication gives
\[
 P_{\rm full}z
 =v+E_s^T\{C_sE_sq-(I_s+C_sW_s)h\}=v.
\]
No inverse of $C_s$ is required. The full action stores the supported contact
matrix and a factor of $I_s+C_sW_s$, with approximately $2s^2$ entries,
and uses two core applications. Building $W_s$ requires $s$ core-response
columns. A diagnostic copy of $W_s$, setup streams, core factors, and Krylov
storage are additional costs and are reported separately.

These are alternative realizations of established inverse identities.
Their floating-point errors and total costs can differ. In particular,
$r<s$ alone implies no ordering between an $nr$ response cache and a
full $s\times s$ compliance realization. A memory comparison must also
consider the implicit rank-$r$ action in
\eqref{eq:implicit-retained-response}. The later measurements identify the
realization used by each comparison method; they make no representation-independent
optimality claim.

\subsection{Maintaining responses through current supported images}
\label{sec:maintained-response}

Repeated equations can share a useful response space even when the weights
of their contact corrections change. Balanced two-level preconditioning
provides an established way to update the operator represented on that space
\citep[Definition~2.1 and Section~2.3.4]{tang2009twolevel}.
Contact-compliance methods likewise separate persistent mechanical mappings
from current contact directions \citep[Section~6]{zeng2023implicit}.
We use these constructions to maintain the responses selected by
\TOGEARI. The following realization exposes the supported data that change
and the complete-equation terms that remain outside the held reference.

Let $j$ index the complete equations $A_j=M_j+C_j+N_j$, where $M_j$ is
the mechanical core, $C_j$ the represented contact tangent and $N_j$ the remaining
terms. Retain the coordinate restriction $E_s$ from
\cref{eq:supported-contact-compliance}. Let $F_\star\in\R^{n\times n}$
be the fixed reference matrix whose factor is prepared once, and let
$C_\star$ be the contact contribution already included in that reference.
The material and complete choices are
\[
 (F_\star,C_\star)=(M_0,0)\quad\hbox{or}\quad
 (F_\star,C_\star)=(A_0,C_0).
\]
Define the signed contact change and current reference by
\[
 \Delta C_j=C_j-C_\star=E_s^T\Delta C_{s,j}E_s,
 \qquad H_j=F_\star+\Delta C_j.
\]
Suppose $F_\star\succ0$ and $H_j\succ0$ for the positivity statements
below. The change $\Delta C_j$ may be indefinite. Choose a full-column-rank
supported load matrix $L_s\in\R^{s\times r}$ and set
\[
 L=E_s^TL_s,\qquad V=F_\star^{-1}L.
\]
TOGEARI supplies these loads through its initial raw or metric selection.
Let $U_0$ factor $C_0=U_0^TU_0$ and choose the basis $Q_{r,0}$ as in
\cref{sec:method,sec:selector-application}, set $R_{r,0}=Q_{r,0}^TU_0$,
and take $L_s=E_sR_{r,0}^T$. The support contains these initial loads.
This instance assumes that $R_{r,0}$ has row rank $r$.
The following proposition applies to any full-column-rank $L_s$ under the
stated reference assumptions; this choice identifies its TOGEARI realization.
A change of basis may normalize $V^TF_\star V$ once, with the same
transformation applied to $L_s$. Define the current supported image, its lift, and its
coarse matrix by
\begin{equation}
 Y_{s,j}=L_s+\Delta C_{s,j}E_sV,\qquad
 Y_j=E_s^TY_{s,j},\qquad G_j=V^TY_j.
\label{eq:maintained-images}
\end{equation}

\begin{proposition}[Supported realization of the balanced action]
\label{prop:maintained-balanced}
Under these hypotheses, $Y_j=H_jV$ and $G_j\succ0$. Define the coarse
inverse $Q_j=VG_j^{-1}V^T$ and the balanced inverse
\begin{equation}
 \mathcal P_j=Q_j+(I-Q_jH_j)F_\star^{-1}(I-H_jQ_j).
\label{eq:maintained-balanced}
\end{equation}
Then $\mathcal P_j\succ0$ and $\mathcal P_jY_j=V$. For an input
$b\in\R^n$, its action is obtained from
\begin{equation}
 \begin{aligned}
 a&=V^Tb, & t&=F_\star^{-1}(b-E_s^TY_{s,j}G_j^{-1}a),\\
 x&=t+VG_j^{-1}(a-Y_{s,j}^TE_st).&&
 \end{aligned}
\label{eq:maintained-action}
\end{equation}
Thus $x=\mathcal P_jb$. Updating the supported contact matrix changes
$Y_{s,j}$ and $G_j$ without a new core solve or a new full response matrix.
\end{proposition}

\begin{proof}
Since $F_\star V=L$, equation~\eqref{eq:maintained-images} gives
$Y_j=F_\star V+\Delta C_jV=H_jV$. Full column rank of $V$ and positivity of $H_j$
imply $G_j=V^TH_jV\succ0$. For any $z\in\R^n$, the quadratic form of
\eqref{eq:maintained-balanced} is the sum of
$(V^Tz)^TG_j^{-1}(V^Tz)$ and
$((I-H_jQ_j)z)^TF_\star^{-1}((I-H_jQ_j)z)$.
If both vanish, $V^Tz=0$ and $(I-H_jQ_j)z=0$; the first equality
reduces the second to $z=0$. Hence $\mathcal P_j$ is SPD.
Moreover, $(I-H_jQ_j)Y_j=0$ and $Q_jY_j=V$, which prove the coarse
identity. Substituting $Y_j=H_jV$ and $Q_j=VG_j^{-1}V^T$ into
\eqref{eq:maintained-balanced} gives \eqref{eq:maintained-action}.
The update assertion follows directly from \eqref{eq:maintained-images}.
\end{proof}

\Cref{alg:maintained-lifecycle} gives the solve-phase lifecycle for the
same construction. Let $b_j$ denote the right-hand side of equation $j$,
and let $x$ denote its linear correction. Its inverse actions use the specified numerical
factorizations; the positive-definiteness and exact coarse-consistency
statements retain the hypotheses of \cref{prop:maintained-balanced}.
Successful numerical factorization alone does not establish them.

\begin{algorithm}[H]
\caption{Maintained \TOGEARI on a compatible equation sequence}
\label{alg:maintained-lifecycle}
\begin{algorithmic}[1]
\Require Initial contact factor $U_0$, prescribed dimension $r$, raw or metric
selection, reference pair $(F_\star,C_\star)$, support $E_s$, and equations $A_jx=b_j$;
metric selection also requires the initial nonsingular symmetric core $M_0$
\If{raw selection}
  \State Form $U_0U_0^T$ and take its $r$ leading eigenvectors $Q_{r,0}$
\Else
  \State Form $U_0M_0^{-1}U_0^T$ and take its $r$ leading eigenvectors $Q_{r,0}$
\EndIf
\State Form $R_{r,0}=Q_{r,0}^TU_0$ and $L_s=E_sR_{r,0}^T$; assume rank $r$
\State Factor $F_\star$ once, reusing the core factor when it is the reference
\State Solve $F_\star V=E_s^TL_s$ once
\For{each compatible current equation $j$}
  \State Form $\Delta C_{s,j}$ from $C_j-C_\star$ on the common support
  \State Form $Y_{s,j}=L_s+\Delta C_{s,j}E_sV$ and
    $G_j=V^TE_s^TY_{s,j}$
  \State Factor $G_j$ with the specified reduced routine; stop this solve
    and report if factorization fails
  \State Solve $A_jx=b_j$ by right-preconditioned FGMRES, using
    \eqref{eq:maintained-action} for each inverse application
  \State Accept the linear direction only after the independently recomputed
    full-current residual satisfies its specified tolerance
  \State Apply the unchanged geometric validity, line-search and nonlinear stopping
    criteria to the resulting trial
\EndFor
\State On a compatible episode reset, restore physical and friction state;
retain the eligible reference factor, responses and structural data
\end{algorithmic}
\end{algorithm}

The raw complete-reference route requires no separate material-core
factorization. Metric selection requires the initial nonsingular symmetric mechanical core
$M_0$ and charges that acquisition as well; the ideal relation between its
responses and complete anchoring is proved below. A prescribed rank is a
construction input, not an adaptive estimate of nonlinear error.

The supported balanced action uses one reference solve per input and retains
$nr$ response entries, $sr$ image entries and an $r$-dimensional coarse factor.
If the initial load support and later contact support lie in a larger
common coordinate set, enlarging $E_s$ leaves $V$ unchanged. The image and
small matrices must be rebuilt on the enlarged support. Coordinate containment
alone gives no retained-space quality guarantee. With a complete anchor,
contact removal gives $\Delta C_j=-C_0$ and still requires the signed update.
When $\Delta C_j=0$, expansion of \eqref{eq:maintained-balanced} and
$Q_jF_\star Q_j=Q_j$ give $\mathcal P_j=F_\star^{-1}$.

Full supported contact receives the same caching opportunity. Define
$Z_{\rm all}=F_\star^{-1}E_s^T$ and $W_{s,\star}=E_sZ_{\rm all}$. With
$q=F_\star^{-1}b$, equation~\eqref{eq:implicit-full-contact} becomes
\[
 H_j^{-1}b
 =q-Z_{\rm all}(I_s+\Delta C_{s,j}W_{s,\star})^{-1}\Delta C_{s,j}E_sq.
\]
This one-core realization stores $ns$ response entries. It supplies the
full-cache comparison in the changing-equation and complete-load experiments.

\paragraph{Core changes and initial response errors.}
Define the complete-equation remainder $\mathcal R_j=A_j-H_j$.
It is $(M_j-M_0)+N_j$ for the material reference and
$(M_j-M_0)+(N_j-N_0)$ for the complete reference.
For an approximate retained response $\widetilde V$, define its initial
response defect $E_0=F_\star\widetilde V-L$ and the image used by the
maintained construction, $\widetilde Y_j=L+\Delta C_j\widetilde V$. Let
$\widetilde{\mathcal P}_j$ denote its realized inverse and define
$T_j=\widetilde{\mathcal P}_j\widetilde Y_j-\widetilde V$.
These definitions give the exact matrix identity
\begin{equation}
 (A_j\widetilde{\mathcal P}_j-I)\widetilde Y_j
 =E_0+\mathcal R_j\widetilde V+A_jT_j.
\label{eq:maintained-defect}
\end{equation}
Indeed, the left side equals
$A_j\widetilde V-\widetilde Y_j+A_jT_j$; expansion of $A_j$ gives
the right side. For a represented right-hand side $\widetilde Y_jc$,
where $c\in\R^r$, its one-action residual is the negative of this
right side times $c$.

The term $E_0$ retains error in the initial response solves. Core changes and
the complete-operator remainder can produce dense image defects despite
unchanged sparse topology. In exact balanced algebra $T_j=0$; approximate
responses, symmetrized coarse matrices and finite actions can contribute to
this term. Equation~\eqref{eq:maintained-defect} identifies those contributions
on the represented image. Errors on omitted directions and global FGMRES
work require separate information. Every solve therefore continues
to test its residual with the current complete $A_j$.

\paragraph{Initial metric directions under complete anchoring.}
The ideal contact-inclusive reference helps explain this choice. Let
$M_0\succ0$, let $U_0\in\R^{m_0\times n}$ factor $C_0=U_0^TU_0$, and
define $F_{\rm ref}=M_0+C_0$. This ideal reference omits the additional
initial remainder $N_0$ in the actual $A_0$. Define
$G^M=U_0M_0^{-1}U_0^T$ and $G^F=U_0F_{\rm ref}^{-1}U_0^T$.

\begin{lemma}[Metric response mapping]
\label{lem:complete-anchor-metric}
The two interaction matrices satisfy
\[
 G^F=G^M(I_{m_0}+G^M)^{-1}.
\]
Let $q_i\in\R^{m_0}$ be a nonzero eigenvector of $G^M$, with positive
eigenvalue $\lambda_i$. Then
\[
 F_{\rm ref}^{-1}U_0^Tq_i
 =(1+\lambda_i)^{-1}M_0^{-1}U_0^Tq_i.
\]
Thus the exact leading metric factor-space eigendirections have the same
ordering, and their physical response spaces agree.
\end{lemma}

\begin{proof}
Multiplying $M_0^{-1}U_0^T(I_{m_0}+G^M)^{-1}$ by $F_{\rm ref}$ gives
$U_0^T$. Left multiplication by $U_0$ gives the interaction identity.
Applying the same relation to $q_i$ gives the response formula. The map
$\lambda\mapsto\lambda/(1+\lambda)$ is strictly increasing on
$[0,\infty)$, and the response scaling is positive.
\end{proof}

This is the spectral mapping induced by the Woodbury identity. It gives a
specific interpretation for contact softening. Let $\lambda_1\ge\cdots
\ge\lambda_n\ge0$ be the eigenvalues of $M_0^{-1/2}C_0M_0^{-1/2}$,
including zeros, and let $H_\alpha=M_0+\alpha C_0$ for $\alpha\ge0$.
Simultaneous diagonalization after scaling by $M_0^{-1/2}$ gives generalized
eigenvalues
\[
 h_i(\alpha)=\frac{1+\alpha\lambda_i}{1+\lambda_i}
\]
for $(H_\alpha,F_{\rm ref})$. Choose $1\le r<n$ with $\lambda_r>0$.
If the balanced space contains the responses
of the $r$ largest positive modes, its projections are diagonal in this
same generalized eigenbasis. The retained diagonal entries become one by
coarse consistency, and the complementary entries remain $h_i(\alpha)$.
This proves the corresponding balanced spectrum directly. For $\alpha<1$,
the initially strongest contact modes become the smallest modes of the
held-factor reference. At complete removal, with a contact nullspace, the
condition number after this correction is $1+\lambda_{r+1}$.

An arbitrary raw space need not be invariant under $G^M$. The actual
complete reference also contains $N_0$; the resolvent identity
\[
 A_0^{-1}L-F_{\rm ref}^{-1}L
 =-A_0^{-1}N_0F_{\rm ref}^{-1}L
\]
retains that effect. These qualifications distinguish the ideal metric
identity from the physical raw-space observations. General LU factors and
symmetrized small Gram matrices are used numerically; the positivity results
retain their stated hypotheses.

\subsection{Frozen inverse-action reference}

The fixed-equation construction below combines selection with the action in
\cref{eq:ZS,eq:woodbury}. It supplies the frozen comparator to the maintained
method. The signed ordering for an indefinite metric core remains empirical.

\begin{algorithm}[H]
\caption{Frozen \TOGEARI: select the interaction space and construct its action}
\label{alg:togeari}
\begin{algorithmic}[1]
\Require Fixed nonsingular core $M$, contact factor $U$, prescribed dimension
$r$, and a raw or metric selector; metric selection requires $M=M^T$
\If{raw selection}
  \State Form $G_0=UU^T$
\Else
  \State Form $G_M=UM^{-1}U^T$ using the core factor
\EndIf
\State Take $Q_r$ as the $r$ leading eigenvectors of the chosen Gram matrix
\State $R_r\gets Q_r^TU$
\State Solve $MZ_r=R_r^T$
\State Form $S_r\gets I_r+R_rZ_r$ and factor it; stop and report if factorization fails
\Function{Apply}{$v$}
  \State $y\gets M^{-1}v$
  \State $w\gets y-Z_rS_r^{-1}R_ry$
  \State \Return $w$
\EndFunction
\end{algorithmic}
\end{algorithm}

\section{Reference spectra and selection guarantees}
\label{sec:theory}

\subsection{Exact spectrum of the metric construction}
Under the SPD specialization, the first question is exact: after $r$ metric
directions are retained, which
parts of $H$ does $P_r$ reproduce, and what discrepancy is left?  Core
whitening makes the answer visible.  It converts the metric selector into an
ordinary singular-vector truncation and places $H$ and $P_r$ in one common
orthogonal basis.

For this subsection, assume $M=M^T\succ0$.  Let $W=UM^{-1/2}$, restore $s=\rank U$ from
\eqref{eq:interaction-rank}, and assume $m<n$, hence $s<n$.
Here $s$ is the factor rank; supported-coordinate counts were used locally
in \cref{sec:response-representations}.  Write
$\spec(H,P_r)$ for the multiset of generalized eigenvalues of the pair
$(H,P_r)$ and $\kappa_2$ for the spectral condition number.  Helmberg's
rank-structured result gives the leading-subspace condition number under the
identification $D_{\mathrm{Helmberg}}=M$ and
$V_{\mathrm{Helmberg}}=U^T$ \citep[Theorem~1]{helmberg2024preconditioned}.
For clarity in the present factor orientation, we record the full generalized
spectrum, whose multiplicities follow directly from simultaneous
diagonalization.

\begin{theorem}[Metric-spectrum specialization]
\label{thm:spectrum}
Assume $M=M^T\succ0$, let $H=M+U^TU$, let
$s=\rank(UM^{-1/2})$, and let $P_r$ be constructed from the $r$ leading
eigenvectors of $UM^{-1}U^T$. If $m<n$, then for $0\le r\le s$ the
generalized spectrum of $(H,P_r)$ is
\begin{equation}
 \spec(H,P_r)=
 \{1\}^{\,n-s+r}
 \cup\{1+\lambda_i:i=r+1,\ldots,s\}.
\label{eq:exact-spectrum}
\end{equation}
Consequently,
\begin{equation}
 \kappa_2(P_r^{-1/2}HP_r^{-1/2})=
 \begin{cases}
  1+\lambda_{r+1},&r<s,\\
  1,&r=s.
 \end{cases}
\label{eq:exact-kappa}
\end{equation}
\end{theorem}
\begin{proof}
Since $WW^T=G_M$, choose the thin singular-value decomposition
\begin{equation}
 W=Q_s\Lambda_s^{1/2}V_s^T,
 \qquad V_s=[v_1,\ldots,v_s]\in\R^{n\times s}.
\label{eq:proof-svd}
\end{equation}
Congruence by $M^{-1/2}$ preserves the generalized eigenvalues of
$(H,P_r)$.  For the complete SPD reference it gives
\[
 M^{-1/2}HM^{-1/2}=I_n+V_s\Lambda_sV_s^T,
\]
so the contact update acts only on $\operatorname{span}(V_s)$.  The definition
$R_r=Q_r^TU$ and \eqref{eq:proof-svd} give
\[
 R_rM^{-1/2}=Q_r^TW=\Lambda_r^{1/2}V_r^T,
 \qquad V_r=[v_1,\ldots,v_r].
\]
Consequently, the whitened preconditioner is
\[
 M^{-1/2}P_rM^{-1/2}=I_n+V_r\Lambda_rV_r^T.
\]
This identity reveals what the correction has accomplished.  For
$i\le r$, both whitened matrices multiply $v_i$ by $1+\lambda_i$, so every
retained direction has generalized eigenvalue one.  For $r<i\le s$, the
reference multiplies $v_i$ by $1+\lambda_i$ and the preconditioner leaves
$v_i$ at its core value, so the generalized eigenvalue is
$1+\lambda_i$.  On the $(n-s)$-dimensional orthogonal complement of
$\operatorname{span}(V_s)$, both matrices are the identity.  There are
therefore $r+(n-s)$ unit eigenvalues and precisely the omitted values shown in
\eqref{eq:exact-spectrum}.  The smallest generalized eigenvalue is one and the
largest is $1+\lambda_{r+1}$ when an omitted direction remains.  This proves
\eqref{eq:exact-kappa}; at $r=s$ the two whitened matrices coincide.
\end{proof}

The spectrum identifies the unresolved modes.  The next corollary measures the
same discrepancy at the level of an inverse action, which is the quantity used
by the preconditioner.

\begin{corollary}[Scaled inverse error]
\label{cor:inverse-error}
Under the hypotheses of \cref{thm:spectrum},
\begin{equation}
 \norm{M^{1/2}(P_r^{-1}-H^{-1})M^{1/2}}_2
 =
 \begin{cases}
  \dfrac{\lambda_{r+1}}{1+\lambda_{r+1}},&r<s,\\[1ex]
  0,&r=s.
\end{cases}
\label{eq:inverse-error}
\end{equation}
\end{corollary}
\begin{proof}
The two whitened matrices in the proof of \cref{thm:spectrum} agree on every
retained direction and on the orthogonal complement of
$\operatorname{span}(V_s)$.  Indeed,
\[
 M^{1/2}P_r^{-1}M^{1/2}
 =\bigl(M^{-1/2}P_rM^{-1/2}\bigr)^{-1},
\]
and the same identity holds for $H$.  Their inverse actions can therefore
differ only on an omitted $v_i$.  There, the scaled inverse of $P_r$ has
eigenvalue one and the scaled inverse of $H$ has eigenvalue
$(1+\lambda_i)^{-1}$.  The difference is
$\lambda_i/(1+\lambda_i)$.  This function is increasing for
$\lambda_i\ge0$, so the spectral norm is attained by the first omitted mode,
$i=r+1$.  No mode is omitted when $r=s$, and the difference is then zero.
\end{proof}

\subsection{Worst omitted interaction and selector optimality}
The exact spectrum explains the metric construction after its eigenvectors have
been chosen.  A second question concerns the choice itself: among all
rank-$r$ factor-space projectors, which one leaves the smallest possible
interaction after core normalization?

Continue under $M=M^T\succ0$.  For an orthogonal factor-space projector
$\Pi\in\R^{m\times m}$ of rank $r$,
define
\begin{equation}
 \eta(\Pi)=\norm{(I_m-\Pi)UM^{-1/2}}_2^2
 =\max_{x\ne0}\frac{\norm{(I_m-\Pi)Ux}_2^2}{x^TMx},
\label{eq:eta}
\end{equation}
and $P_\Pi=M+U^T\Pi U$.

\begin{proposition}[Optimal omitted interaction]
\label{prop:selector-optimality}
For $0\le r\le s$, the metric projector $\Pi_r=Q_rQ_r^T$ minimizes $\eta(\Pi)$ over all
rank-$r$ orthogonal projectors.  Its value is $\lambda_{r+1}$ for $r<s$ and
zero for $r=s$.  Every such projector also satisfies
\begin{equation}
 \kappa_2(P_\Pi^{-1/2}HP_\Pi^{-1/2})\le1+\eta(\Pi).
\label{eq:any-selector-bound}
\end{equation}
\end{proposition}
\begin{proof}
By \eqref{eq:eta}, $\eta(\Pi)$ is the squared spectral norm of the part of
$W=UM^{-1/2}$ left outside $\operatorname{range}(\Pi)$.  The singular-value
min--max principle makes the span of the first $r$ left singular vectors of
$W$ optimal.  Those vectors are the columns of $Q_r$, and the squared first
omitted singular value is $\lambda_{r+1}$.

For the condition bound, set $E_\Pi=U^T(I_m-\Pi)U$, so that
$H=P_\Pi+E_\Pi$.  The definition of $\eta$ gives, for every $x\in\R^n$,
\[
 x^TE_\Pi x=\norm{(I_m-\Pi)Ux}_2^2
 \le \eta(\Pi)x^TMx
 \le \eta(\Pi)x^TP_\Pi x.
\]
Hence $P_\Pi\preceq H\preceq(1+\eta(\Pi))P_\Pi$.  The generalized
eigenvalues lie in $[1,1+\eta(\Pi)]$, which proves
\eqref{eq:any-selector-bound} and recovers the arbitrary-subspace estimate in
this factor orientation \citep[Theorem~2]{helmberg2024preconditioned}.
\end{proof}

Thus \TOGEARIM answers the selection question directly: it minimizes the
largest core-normalized contact energy left outside the retained factor space.
Any other projector has the same form of posterior condition bound once its
value of $\eta$ is evaluated.

\subsection{Exact endpoint for an arbitrary retained space}

Let $M\succ0$, let $U\in\R^{m\times n}$ with $m<n$, and let
$Q=[Q_1,Q_2]\in\R^{m\times m}$ be orthogonal, with $Q_1$ having $r$ columns.
Set $P=M+U^TQ_1Q_1^TU$, $H=M+U^TU$, and partition the response Gram matrix
in this basis as
\[
Q^TUM^{-1}U^TQ=
\begin{bmatrix}G_{11}&G_{12}\\G_{21}&G_{22}\end{bmatrix}.
\]
Define the omitted response matrix
\begin{equation}
S_{\rm omit}=G_{22}-G_{21}(I_r+G_{11})^{-1}G_{12}.
\label{eq:omitted-response-schur}
\end{equation}

\begin{proposition}[Arbitrary-selector response endpoint]
\label{prop:arbitrary-endpoint}
The matrix $S_{\rm omit}$ is positive semidefinite, and
\[
\spec(H,P)=\{1\}^{\,n-m+r}
\cup\{1+\mu_i(S_{\rm omit}):i=1,\ldots,m-r\},
\]
where $\mu_i$ are eigenvalues counted with multiplicity. Thus the spectral
condition number is $1+\lambda_{\max}(S_{\rm omit})$ when $r<m$.
At $r=m$ it equals one.
\end{proposition}

\begin{proof}
Write $V=Q_2^TU$, so $H=P+V^TV$. The Woodbury identity gives
\[
VP^{-1}V^T
=G_{22}-G_{21}(I_r+G_{11})^{-1}G_{12}
=S_{\rm omit}.
\]
This representation proves positive semidefiniteness. Congruence by
$P^{-1/2}$ transforms the generalized pair into
$I_n+(VP^{-1/2})^T(VP^{-1/2})$ and $I_n$. The two Gram products
$(VP^{-1/2})^T(VP^{-1/2})$ and $VP^{-1}V^T$ have the same nonzero
eigenvalues; padding by $n-m+r$ zeros gives the displayed spectrum.
Because $m<n$, there is at least one unit eigenvalue. The largest is
$1+\lambda_{\max}(S_{\rm omit})$. If $r=m$, $V$ has no rows and $H=P$.
\end{proof}

This is an exact Woodbury--Schur specialization of the rank-structured
preconditioning framework \citep{helmberg2024preconditioned}. It sharpens
the upper estimate based only on the omitted block $G_{22}$ and applies to
raw, metric, or externally supplied spaces. Computing the complete response
Gram matrix still has a cost; the proposition alone does not yield an inexpensive
adaptive estimator. The endpoint concerns the SPD reference $H$, whereas the
complete equation also contains $N$.

\subsection{Spectral improvement and a particular right-hand side}

The spectral statements concern the complete generalized spectrum. A Krylov
residual also depends on the initial right-hand side and the norm in which
it is minimized. These distinctions are central in the analysis of GMRES
\citep{greenbaum1996any,carson2024understanding}. They remain relevant when
the reference matrix is SPD and the correction uses exact eigenvectors.
The following elementary example makes the limitation explicit within
the hypotheses of \cref{thm:spectrum}.

\begin{remark}[A partial correction can worsen the first residual]
\label{rem:partial-cluster}
Let $t>s>1$, let $M=I_3$, and define
\[
U=\begin{bmatrix}\sqrt{t-1}&0&0\\0&\sqrt{s-1}&0\end{bmatrix},
\qquad H=\diag(t,s,1),\qquad b=(1,1,0)^T.
\]
The rank-zero and leading rank-one preconditioners are $P_0=I_3$ and
$P_1=\diag(t,1,1)$. For $j\in\{0,1\}$, define the optimal first-step
relative residual from a zero initial guess by
\[
\rho_j=\min_{\alpha\in\R}
 \frac{\norm{b-\alpha HP_j^{-1}b}_2}{\norm b_2}.
\]
For any nonzero vector $w$, the scalar minimizing $\norm{b-\alpha w}_2$
is $\alpha=b^Tw/(w^Tw)$. Applying this identity to
$HP_0^{-1}b=(t,s,0)^T$ and $HP_1^{-1}b=(1,s,0)^T$ gives
\[
\rho_0=\frac{t-s}{\sqrt{2(t^2+s^2)}},\qquad
\rho_1=\frac{s-1}{\sqrt{2(1+s^2)}}.
\]
At $t=1000$ and $s=990$, the full spectral condition number decreases
from $1000$ to $990$, while $\rho_0\simeq0.00503$ and
$\rho_1\simeq0.70639$. The invariant subspace containing $b$ initially
sees two nearby eigenvalues; the partial correction moves only one of
them to one. Thus the exact spectral improvement in
\cref{thm:spectrum} implies no monotone first-step residual improvement
for a fixed right-hand side.
\end{remark}

The example uses $N=0$ and exact arithmetic. It therefore also distinguishes
RHS-dependent behavior from errors in the core solve or from the terms
outside the SPD reference. The experiment in
\cref{sec:high-contact-qualification} examines these possibilities separately
on an actual contact-bearing equation.

\subsection{Use in the complete Newton equation}
The computation ultimately solves $A=H+N$, which can be indefinite or
nonsymmetric. The spectrum above describes the SPD pair $(H,P_r)$; the full
FGMRES process also depends on $N$. For every reported solve, the
selected right preconditioner is applied to the unchanged $A$.  Let
$x\in\R^n$ be a candidate iterate, and let $\varepsilon_{\mathrm{abs}}$ and
$\varepsilon_{\mathrm{rel}}$ be the absolute and relative residual tolerances.
The iterate is accepted only when
\begin{equation}
 \norm{b-Ax}_2\le
 \max\{\varepsilon_{\mathrm{abs}},
       \varepsilon_{\mathrm{rel}}\norm{b}_2\}
\label{eq:residual-test}
\end{equation}
under an independently evaluated complete operator action. This verifies the
stated normwise residual criterion for the original equation $Ax=b$.
Forward-error, field-of-values, and nonlinear-trajectory conclusions require
their own analyses.

\subsection{Condensed-volumetric scaling and zero contact}
The final algebraic question is how the selector changes when the condensed
volumetric term is strengthened on one frozen equation.  This isolates a matrix
effect at a fixed state; re-equilibration at a new material parameter is a
separate experiment.

\begin{proposition}[Frozen volumetric monotonicity]
Let
\begin{equation}
 M(\alpha)=K_0+\alpha K_{\mathrm{vol}},
 \qquad \alpha\ge0,
\label{eq:Malpha}
\end{equation}
with $K_0$ and $U$ fixed, $K_{\mathrm{vol}}\succeq0$, and
$M(\alpha)\succ0$ on the interval.  If $\alpha_2\ge\alpha_1$, then
\begin{equation}
 UM(\alpha_2)^{-1}U^T
 \preceq UM(\alpha_1)^{-1}U^T.
\label{eq:pressure-monotone}
\end{equation}
Every ordered eigenvalue of the metric Gram matrix is therefore nonincreasing
with $\alpha$.
\end{proposition}
\begin{proof}
Positive semidefiniteness of $K_{\mathrm{vol}}$ gives
$M(\alpha_1)\preceq M(\alpha_2)$.  Inversion reverses the Loewner order for
SPD matrices, and congruence by $U$ preserves it, which yields
\eqref{eq:pressure-monotone}.  Monotonicity of each ordered eigenvalue follows
from the symmetric-matrix min--max principle.
\end{proof}

The proposition covers frozen algebraic scaling. A change in bulk modulus or
Poisson ratio also changes the pressure-geometric curvature, and re-equilibration can alter
the state, geometry, and contact factor; those coupled changes lie beyond its
hypotheses.

If $U$ has zero rows, both Gram matrices are empty, the selected rank is zero,
and $P_0=M$. A conforming implementation skips the contact response,
eigensolve, and retained-basis allocation in that case.

\section{Current-equation convergence and repeated reuse}
\label{sec:convergence-reuse}

The reference spectrum describes the interactions left by a correction.
For a maintained space, the current contact may also soften or rotate,
and the complete equation contains terms outside the reference. We connect
these effects to the inner solve and then to a local nonlinear step.
The analysis specializes established balanced-preconditioner spectral
relations \citep[Section~3.1]{tang2009twolevel}, GMRES polynomial bounds
\citep[Section~1 and Theorem~2.1]{liesen2020ideal}, and inexact-Newton
convergence \citep{dembo1982inexact,eisenstat1996forcing}.
Its sufficient conditions have a different role from the measured
complete-residual acceptance criterion.

\subsection{The omitted complement under signed contact changes}

Fix one equation and suppress its index in
\cref{sec:maintained-response}. Let $p$ be the number of supported
coordinates and write the existing restriction as $E_s\in\R^{p\times n}$,
with independent rows. Assume $1\le p<n$.
Let $D_c\in\R^{p\times p}$ be the symmetric supported contact change.
Thus
\[
 H=F_\star+E_s^TD_cE_s,\qquad F_\star\succ0,\qquad H\succ0.
\]
Let $L_p\in\R^{p\times r}$ have full column rank, with $0\le r\le p$,
and set $V=F_\star^{-1}E_s^TL_p$. The balanced inverse
$\mathcal P$ is defined by \eqref{eq:maintained-balanced} using this
reference, current $H$, and response space $V$.
Define the supported reference compliance and the whitened load by
\[
 \Omega=E_sF_\star^{-1}E_s^T,\qquad W_L=\Omega^{1/2}L_p.
\]
The matrix $\Omega$ is SPD. Choose an orthogonal matrix
$O=[Z,Z_\perp]\in\R^{p\times p}$ whose first $r$ columns span $W_L$.
Partition the supported reference in this basis:
\begin{equation}
 O^T(I_p+\Omega^{1/2}D_c\Omega^{1/2})O
 =\begin{bmatrix}K_c&B_c^T\\B_c&T_c\end{bmatrix},
 \qquad
 S_c=T_c-B_cK_c^{-1}B_c^T.
\label{eq:balanced-supported-schur}
\end{equation}
Here $K_c\in\R^{r\times r}$ and $S_c\in\R^{(p-r)\times(p-r)}$.
At $r=0$, $S_c$ is the entire supported matrix; at $r=p$,
the omitted block is empty.

\begin{proposition}[Supported complementary spectrum]
\label{prop:balanced-complement}
Under these hypotheses, $K_c$ and $S_c$ are SPD whenever nonempty, and
\begin{equation}
 \spec(\mathcal P^{1/2}H\mathcal P^{1/2})
 =\{1\}^{\,n-p+r}\cup\spec(S_c).
\label{eq:balanced-complement-spectrum}
\end{equation}
For $r<p$, the smallest and largest eigenvalues are
\[
 a=\min\{1,\lambda_{\min}(S_c)\},\qquad
 b=\max\{1,\lambda_{\max}(S_c)\}.
\]
For $r=p$ they are $a=b=1$.
\end{proposition}
\begin{proof}
The matrix $X=F_\star^{-1/2}E_s^T\Omega^{-1/2}$ satisfies $X^TX=I_p$.
It identifies the supported part of $F_\star^{-1/2}HF_\star^{-1/2}$;
the orthogonal complement of $X$ is an identity block.
In the basis $XO$, the whitened response space is the first block.
Expanding the balanced inverse on this space gives
\[
 \begin{bmatrix}
 K_c^{-1}+K_c^{-1}B_c^TB_cK_c^{-1}&-K_c^{-1}B_c^T\\
 -B_cK_c^{-1}&I_{p-r}
 \end{bmatrix}.
\]
Its inverse is
\[
 \begin{bmatrix}
 K_c&B_c^T\\
 B_c&I_{p-r}+B_cK_c^{-1}B_c^T
 \end{bmatrix}.
\]
This matrix and the supported reference in
\eqref{eq:balanced-supported-schur} have the same unit lower triangular
block factor
$\left[\begin{smallmatrix}I_r&0\\B_cK_c^{-1}&I_{p-r}\end{smallmatrix}\right]$.
Their middle diagonal factors are respectively
$\diag(K_c,I_{p-r})$ and $\diag(K_c,S_c)$.
Their generalized eigenvalues are therefore $r$ ones and the
eigenvalues of $S_c$. The unsupported complement contributes $n-p$
further ones. Positive definiteness follows from the SPD supported
reference and its Schur complement. This proves the spectrum and endpoints.
The empty-block cases follow by omitting the corresponding factors.
\end{proof}

The coupling term $B_cK_c^{-1}B_c^T$ matters even when the retained
coarse matrix is well conditioned. For example, take $F_\star=I_3$,
let $E_s$ restrict the first two coordinates, and retain $L_p=(1,0)^T$.
With initial contact $0.99I_2$ and current contact
$0.99\left[\begin{smallmatrix}1&1\\1&1\end{smallmatrix}\right]$,
both contacts are PSD, while their difference is signed.
The supported $H$ is
$\left[\begin{smallmatrix}1&0.99\\0.99&1\end{smallmatrix}\right]$.
Thus $K_c=1$, $S_c=1-0.99^2=0.0199$, and the condition number in
\eqref{eq:balanced-complement-spectrum} is about $50.25$.
Coarse positivity alone does not yield a uniform omitted-mode bound.

The proposition requires only small supported matrices once $\Omega$
is available. Acquiring $\Omega$ generally requires $p$ reference
responses; the $r$ retained responses do not provide it automatically.
It is an analysis quantity, with acquisition cost included when measured.
Empty contact support recovers $\mathcal P=F_\star^{-1}$ directly.

\subsection{From the reference to the complete linear equation}

Let $A=H+\mathcal R$ be the current complete matrix, and let
$\widetilde{\mathcal P}=\mathcal P+\mathcal E$ be a fixed linear
inverse action during one inner solve. The term $\mathcal E$ can
represent fixed approximate response or coarse data. Define
\begin{equation}
 \begin{aligned}
 \delta_{\mathcal R}
 &=\norm{\mathcal P^{1/2}\mathcal R\mathcal P^{1/2}}_2,\\
 \delta_{\mathcal E}
 &=\norm{\mathcal P^{1/2}A\mathcal E\mathcal P^{-1/2}}_2,
 \qquad \delta=\delta_{\mathcal R}+\delta_{\mathcal E}.
 \end{aligned}
\label{eq:complete-convergence-defects}
\end{equation}
These norms concern every direction, including the omitted space.
Equation~\eqref{eq:maintained-defect} describes the defect only on the
represented current image.
For a right-hand side $b_{\mathrm{lin}}\in\R^n$, let $x_\ell$ be the
linear iterate after $\ell$ columns and define its residual by
$r_\ell=b_{\mathrm{lin}}-Ax_\ell$. Assume $r_0\ne0$; a zero
initial residual requires no iteration.

\begin{proposition}[A sufficient current-equation convergence bound]
\label{prop:complete-gmres-bound}
Let $a,b$ be the endpoints in \cref{prop:balanced-complement}.
Suppose $\delta<a$, and define
\[
 \alpha=a-\delta,\qquad \beta=b+\delta,\qquad
 q=\sqrt{1-(\alpha/\beta)^2},\qquad
 \gamma=\sqrt{\kappa_2(\mathcal P)}.
\]
In exact arithmetic, after $\ell$ unrestarted GMRES columns with right
action $\widetilde{\mathcal P}$, the true residual satisfies
\begin{equation}
 \frac{\norm{r_\ell}_2}{\norm{r_0}_2}
 \le \min\{1,\gamma q^\ell\}.
\label{eq:complete-gmres-bound}
\end{equation}
For restart length $m$, the same argument gives the per-cycle factor
$\min\{1,\gamma q^m\}$. In particular, $\gamma q^m<1$ implies
geometric reduction over restart cycles.
\end{proposition}
\begin{proof}
Set $\mathcal B=\mathcal P^{1/2}A\widetilde{\mathcal P}\mathcal P^{-1/2}$.
It equals the SPD reference $\mathcal P^{1/2}H\mathcal P^{1/2}$
plus the two perturbations in \eqref{eq:complete-convergence-defects}.
Hence its symmetric part is at least $\alpha I_n$ and
$\norm{\mathcal B}_2\le\beta$. For $\omega=\alpha/\beta^2$,
expansion of the squared norm gives, for every $v\in\R^n$,
\[
 \norm{(I_n-\omega\mathcal B)v}_2^2
 \le (1-2\omega\alpha+\omega^2\beta^2)\norm v_2^2
 =q^2\norm v_2^2.
\]
The polynomial $(1-\omega z)^\ell$ has value one at zero.
Similarity between $\mathcal B$ and $A\widetilde{\mathcal P}$
therefore bounds its actual-norm residual by $\gamma q^\ell$.
GMRES minimizes over a space containing that polynomial candidate
and cannot increase the initial residual. This proves
\eqref{eq:complete-gmres-bound}. Apply the same estimate to the
initial residual of each restart cycle to obtain the final assertion.
\end{proof}

The factor $\gamma$ retains the change from an energy-scaled norm to
the Euclidean norm used by the implementation. It cannot generally be
discarded at restarts. A flexible recurrence with a fixed linear action
has the same exact-arithmetic search space; genuinely varying inner
actions and finite-precision Arnoldi errors require additional analysis.
The maintained numerical method updates its correction between equations.
The reference bounds are sufficient conditions, and may be conservative.
For instance, $A=\diag(1,1,-1)$ with $H=\mathcal P=I_3$ fails
$\delta<a$, although the polynomial $1-z^2$ annihilates $A$.
An unavailable bound therefore gives no conclusion about failure to solve.

\subsection{Inheritance by a damped nonlinear step}

Let $f:\mathcal U\subset\R^n\to\R^n$ be one fixed mechanical
residual on an open neighborhood of a root $x_\star$.
Write $\mathcal J(x)=f'(x)$ for its derivative, and assume that $\mathcal J$
is Lipschitz continuous and $\mathcal J(x_\star)$ is nonsingular.
The residual may represent static force balance or one implicit
time-step equation with fixed committed history.
At an iterate $x$, let $A$ be the assembled tangent, $s$ the computed
direction, and $\theta\in(0,1]$ the accepted step fraction. Define
\[
 e=As+f(x),\qquad d=(\mathcal J(x)-A)s.
\]
Here $e$ is the inner linear error and $d$ is the directional tangent
defect. When $f(x)\ne0$, let
$\eta=\norm e_2/\norm{f(x)}_2$ and
$\zeta=\norm d_2/\norm{f(x)}_2$.
Subscript $k$ will index the accepted nonlinear iterations.

\begin{proposition}[Damped local convergence with tangent error]
\label{prop:damped-newton-inheritance}
If $L_J$ bounds the Lipschitz constant of $\mathcal J$ along the candidate
segment, then
\begin{equation}
 \norm{f(x+\theta s)}_2
 \le [1-\theta+\theta(\eta+\zeta)]\norm{f(x)}_2
       +\frac{L_J}{2}\theta^2\norm s_2^2.
\label{eq:damped-newton-envelope}
\end{equation}
Suppose the accepted iterations satisfy
$\eta+\zeta\le\overline\eta<1$ and
$\theta\ge\theta_{\min}>0$ uniformly near $x_\star$.
For a sufficiently close initial state, they converge locally to
$x_\star$ with geometric residual reduction.
Sufficient additional conditions for superlinear reduction are
$1-\theta_k+\eta_k+\zeta_k\to0$; reduction is quadratic if this
quantity is $O(\norm{f(x_k)}_2)$.
\end{proposition}
\begin{proof}
Since $f(x)+\mathcal J(x)s=e+d$, the linear part of the Taylor expansion is
$(1-\theta)f(x)+\theta(e+d)$.
The integral remainder has norm at most
$L_J\theta^2\norm s_2^2/2$, proving
\eqref{eq:damped-newton-envelope}.
Choose a root neighborhood with $\norm{\mathcal J(x)^{-1}}_2\le c_J$.
Then
$\norm s_2\le c_J(1+\overline\eta)\norm{f(x)}_2$.
The coefficient of the first-order residual is at most
$q_0=1-\theta_{\min}(1-\overline\eta)<1$.
In a smaller neighborhood the quadratic term is at most
$(1-q_0)\norm{f(x)}_2/2$, giving a contraction factor
$(1+q_0)/2<1$.
Choose the initial residual small enough that the resulting geometric
sum of step lengths stays inside the original neighborhood.
The iterates then remain there, have a limit with zero residual, and
converge to its unique local root. Dividing
\eqref{eq:damped-newton-envelope} by the current residual proves
the superlinear assertion under its additional condition; the
$O(\norm{f(x_k)}_2)$ condition gives the quadratic assertion.
\end{proof}

The conditions for the faster rates are sufficient, and the separate
norms can overestimate cancellation between $e$ and $d$.
A fixed allowed inner tolerance does not by itself imply quadratic
convergence. Persistent damping also remains visible: for the scalar
residual $f(x)=x+x^2/2$, an exact Newton direction with fixed
$\theta=0.9$ has residual ratio tending to $0.1$.
Projected tangents and lagged models require their actual $d$ to be
accounted for. If the residual model changes after the trial,
the difference between the new and old residual at the candidate is
an additional term in \eqref{eq:damped-newton-envelope}.
The proposition does not assert finite contact-model identification
or success within a prescribed finite backtracking budget.

\subsection{Uniform reuse over an episode family}

Let an index $\nu$ label the equations visited by a family of load
histories or episodes. Suppose they share $F_\star$ and the retained
load space, and let their constants in
\cref{prop:complete-gmres-bound} satisfy
$q_\nu\le\overline q<1$ and $\gamma_\nu\le\overline\gamma<\infty$.
For a common restart length $m$ with
$\overline\gamma\,\overline q^m<1$, the proposition gives the same
geometric per-cycle bound for every $\nu$.
Uniform bounds on the derivative Lipschitz constants, on
$\|\mathcal J(x)^{-1}\|_2$, and on $\eta+\zeta<1$, together with a
common positive lower step fraction and common admissible root-neighborhood
radius in \cref{prop:damped-newton-inheritance}, likewise give a common
local nonlinear contraction for sufficiently close initial states.
These statements follow by substituting the uniform constants
in the pointwise bounds.

A compatible reset restores physical and friction histories while
allowing the reference factor and response data to remain resident.
It can return successive episodes to a common useful operator region.
No episode count appears in the bounds, although the region must
contain the states actually visited by different policies.
Changed mass, time step, boundary constraint sets or geometry require
their own compatibility assessment. A bounded common potential
support is also needed for uniform cache storage.
Reusable reference spaces across parameter families have direct
precedents \citep{kadeethum2026online}.

These estimates explain which properties must remain controlled for
reuse to stay effective. The full-residual checks in the experiments
verify the solved equations, though they do not cover every global norm
or uniform smoothness hypothesis above. The time benefit is assessed
with preparation, updates, resets and solve work included.
Convergence of a learned policy does not follow.

\subsection{Theory, implementation and observed scope}
\label{sec:theory-evidence-map}

\Cref{tab:theory-evidence-map} collects the assumptions used by the
preceding results and the distinct checks performed by the implementation.
It separates a conditional convergence explanation from a residual
observation. Measured residuals and unmeasured global norms, positive-definiteness
hypotheses, and local smoothness assumptions occupy different rows.

\begin{table}[htbp]
\centering\small
\caption{Applicability of the reference and convergence results.
The physical observations establish the checks in the right column
within their reported scope.}
\label{tab:theory-evidence-map}
\begin{tabularx}{\textwidth}{@{}>{\raggedright\arraybackslash}p{0.23\textwidth}
 >{\raggedright\arraybackslash}X >{\raggedright\arraybackslash}X@{}}
\toprule
Result or construction & Mathematical requirements & Numerical scope\\
\midrule
Woodbury inverse action
& Nonsingular core and reduced system; no general SPD requirement
(\cref{eq:woodbury}).
& GPU LU core action and Cholesky of a symmetrized reduced matrix
are checked separately.\\
Metric spectrum and optimality
& A fixed symmetric positive-definite core and exact actions
(\cref{thm:spectrum,prop:selector-optimality}).
& Primary global-core SPD was not recorded. Its work counts remain
direct observations; the spectral certificate is conditional.\\
Balanced current images and complement
& SPD reference and current contact-updated reference; full-rank
retained loads (\cref{prop:maintained-balanced,prop:balanced-complement}).
& Current supported images are updated. The retained responses alone
do not measure the omitted complement.\\
Complete-equation inner convergence
& Full-space drift and inverse-action bounds, with a fixed action
during each inner solve (\cref{prop:complete-gmres-bound}).
& Complete residuals are independently recomputed. They do not
establish all full-space perturbation bounds.\\
Damped local Newton and uniform reuse
& A fixed smooth residual near a nonsingular root, controlled tangent
error and damping; uniform bounds for an episode family.
& Numerical convergence and reset histories are observed. Uniform smoothness,
tangent-error bounds and a positive lower step fraction are not
established for every visited physical state.\\
\bottomrule
\end{tabularx}
\end{table}

\section{Numerical evidence}
\label{sec:protocol}
\label{sec:results}

The evidence asks whether the chosen interactions help the same core,
whether compression retains that benefit, and whether maintaining the
responses preserves useful complete-work economics.
Fixed-equation comparisons isolate augmentation and compression.
Changing-equation comparisons test maintenance of numerical preparation. Larger-contact and
frictional cases then delimit the same construction.

\subsection{Physical problems and comparison contract}

The primary problem is a pneumatic finger contacting a deformable cup at
$150$~Pa. Its $325{,}260$ free displacement coordinates come from $66{,}315$
quadratic tetrahedra: $58{,}633$ condensed Tet10/P0 and $7{,}682$ plain Tet10.
The three material groups have shear/bulk moduli, in MPa, of
$(0.0334,3.3333)$, $(0.6712,33.3333)$ and $(10.7143,50.0)$ on
$26{,}613$, $7{,}682$ and $32{,}020$ elements.
The finger base and minimum-$y$ cup support are fixed; nine closed chamber
surfaces carry follower pressure.
The contact proxy has $64$ triangles per six-node quadratic triangular (TRI6) boundary face
and distance $0.01$~mm.
The primary frozen factor has $94$ rows, $2479$ nonzeros and $60$ supported
free coordinates from $36$ normal pairs.
It is frictionless. Its RHS norm $0.5032158$ defines a sampled linear
equation, independently of equilibrium.

\begin{figure}[H]
\centering
\includegraphics[width=\linewidth]{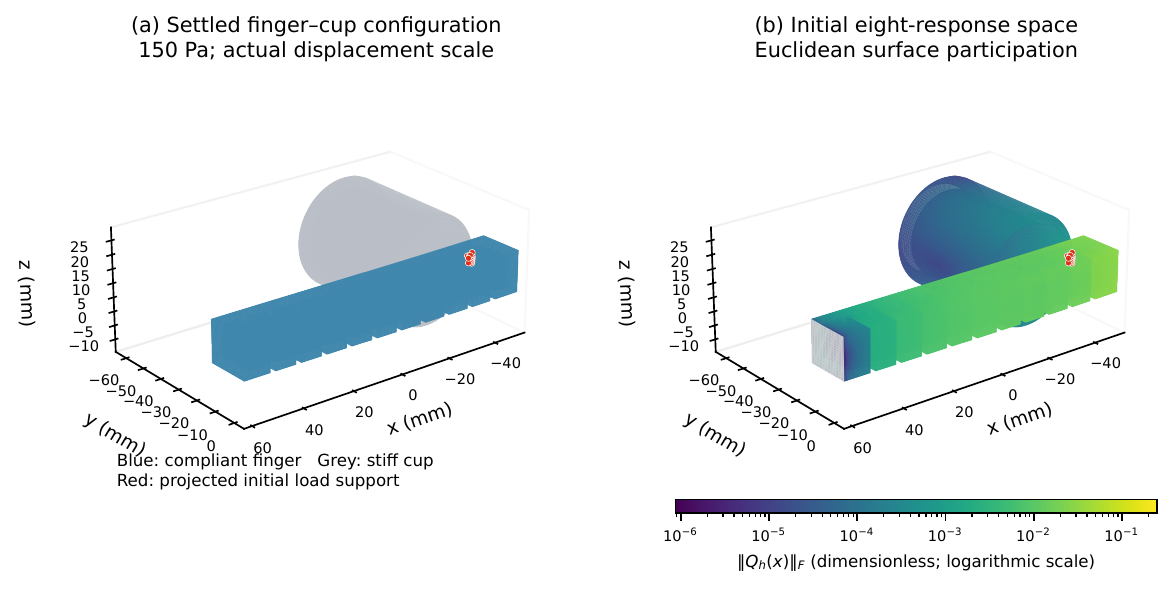}
\caption{The same finger/cup model in a separately converged $150$-Pa state,
with actual displacement scale. Blue and grey identify the compliant finger
and stiff cup. Red markers project the $21$ initial load-support nodes above
the surfaces, including occluded nodes. For this figure, let $Q$ be a
Euclidean-orthonormal basis of the initial eight-response space and let
$Q_h(x)\in\R^{3\times8}$ interpolate its nodal vectors at a surface point.
Panel (b) shows $\|Q_h(x)\|_F$, which is invariant under an orthogonal
basis change. The logarithmic scale covers all positive display samples;
zero values are grey. This is Euclidean participation, not an energy fraction
or mesh-independent indicator.
Four display facets approximate each quadratic boundary face.}
\label{fig:retained-response-scene}
\end{figure}

The larger four-finger/cup problem has $865{,}572$ stored and $856{,}206$
free coordinates. Its controls are $750$~Pa, radial closure $1.25$~mm,
zero vertical command, friction coefficient $\mu_f=0.3$ and regularization
length $\varepsilon_f=10^{-3}$~mm.
The separately staged high-contact equation and difficult late equation
have different captured states; their exact dimensions and protocols
accompany the supporting comparisons.
The primary core is $K_0+K_{\mathrm{vol}}$.
The larger complete trajectories instead use the current complete
noncontact core $\overline M$, including its pressure and noncontact
curvature, with normal and lagged-friction contact in the root.

All comparisons use float64 arithmetic on a Tesla V100S-PCIE-32GB with
cuDSS 0.8.0 and one CPU core per process.
The GPU right-FGMRES recurrence uses two modified Gram--Schmidt
passes and independently recomputes the complete residual.
CPU-controlled observations are identified separately.
For a nonzero RHS $b$, the relative linear residual is
$\rho(x)=\|b-Ax\|_2/\|b\|_2$.
A forcing value $\eta$ permits $\rho(x)\le\eta$.
Numerical convergence of the nonlinear solve requires an absolute residual
at most $10^{-6}$ or a relative residual at most $10^{-8}$, together with
validity and friction consistency.

Warm times for the fixed-equation GPU solves include terminal solution transfer and use the
stated within-process medians. Complete-load clocks charge model initialization, factorization, acquisition, updates, every attempted target,
resets, writing results and releasing solver resources; process startup before that clock is excluded.
The complete-load prototype returns its direction to CPU memory.
A solve stopped at its iteration limit remains an incomplete observation.
These timing scopes and full-residual checks are retained when comparisons
move from a fixed equation to a nonlinear history.

\subsection{What the selected interactions contribute}
\label{sec:fixed-core-benefit}
\label{sec:e1}
\label{sec:e2}
\label{sec:resident-core-comparison}

The extracted primary components reconstruct the complete free operator
to relative Frobenius discrepancy $6.53\times10^{-17}$. The contact root
reconstructs its projected tangent to below $5\times10^{-13}$.
At retained dimension eight, metric and raw selection both generate five
Krylov columns; core-only generates eleven. Dimension-matched row-norm and
deterministic-random selections require twelve. Thus the content of the
subspace matters in this comparison.

\begin{table}[H]
\centering\small
\caption{Matched GPU actions on the primary equation. Each entry is
a warm median of five subsequent solves. The repeated order tests use
one shared factor; the full-contact row comes from the compatible GPU
dimension sweep.}
\label{tab:resident-core-comparison}
\begin{tabular}{@{}lrrr@{}}
\toprule
Method & Retained dimension & Krylov columns & Warm median (ms)\\
\midrule
Core only, first & 0 & 11 & 81.062\\
Raw, first & 8 & 5 & 37.538\\
Metric, first & 8 & 5 & 37.571\\
Metric, reverse order & 8 & 5 & 37.506\\
Raw, reverse order & 8 & 5 & 37.535\\
Core only, reverse order & 0 & 11 & 80.044\\
Full metric correction & 94 & 5 & 39.644\\
\bottomrule
\end{tabular}
\end{table}

The large reduction from approximately $80$ to $37.5$~ms measures contact
augmentation. Full contact also takes five columns. Compression supplies
a further approximately $5.4\%$ warm reduction in the GPU sweep.
The dimension-eight and full response arrays occupy $20.817$ and
$244.596$~MB, respectively; those sizes exclude other GPU allocations.

Common core and complete-operator setup cost $22.349$ and $18.493$~s.
Metric selection costs approximately $0.108$~s; raw selection costs
$0.0379$~s on first use and $0.00176$~s after warm-up. These costs matter
whenever a correction is rebuilt. The largest independently recomputed
relative residuals are $8.860\times10^{-10}$ for core-only and
$4.291\times10^{-10}$ for rank eight.

The first omitted core-response value falls from $3.071$ at dimension six
to $0.07746$ at eight and $0.00572$ at twelve. Conditional on $M\succ0$,
the dimension-eight reference condition number is approximately $1.07746$,
and the core-scaled inverse-error norm is approximately $0.07189$.
The complete dimension sweep is given in \cref{tab:rank}.
The metric and raw spaces have principal angles from
$3.47\times10^{-4}$ to $0.232$ radians. Their close alignment is consistent
with the common work count; unique practical necessity of the metric selector
is unestablished by this case.

\begin{figure}[H]
\centering
\begin{tikzpicture}
\begin{axis}[
 at={(0,0)},anchor=south west,width=0.48\linewidth,height=5.6cm,
 xmin=0,xmax=24,ymin=4,ymax=12,
 xtick={0,4,8,12,16,24},ytick={5,6,8,10,11},
 xlabel={Retained dimension $r$},ylabel={Krylov columns},
 title={(a) Krylov work plateaus at eight},
 label style={font=\small},tick label style={font=\small},
 title style={font=\small},grid=major,grid style={gray!15}]
\addplot[blue!65!black,thick,mark=*] coordinates {
 (0,11) (2,10) (4,8) (6,6) (8,5) (12,5) (16,5) (24,5)};
\addplot[gray,densely dashed] coordinates {(0,5) (24,5)};
\node[anchor=north east,align=right,font=\scriptsize] at (rel axis cs:0.96,0.94)
 {Full contact: 94 rows\\also five columns};
\end{axis}
\begin{axis}[
 at={(7.6cm,0)},anchor=south west,width=0.48\linewidth,height=5.6cm,
 xmin=0,xmax=24,ymode=log,
 xtick={2,4,8,12,16,24},xlabel={Retained dimension $r$},
 ylabel={First omitted $\lambda$},
 title={(b) Conditional reference quantity},
 label style={font=\small},tick label style={font=\small},
 title style={font=\small},grid=major,grid style={gray!15}]
\addplot[orange!75!black,thick,mark=square*] coordinates {
 (2,559.2) (4,11.44) (6,3.071) (8,0.07746)
 (12,0.005716) (16,0.001426) (24,0.000003429)};
\end{axis}
\end{tikzpicture}
\caption{Interaction truncation and observed work on the primary frozen
equation. The left panel shows the complete-residual Krylov work; the
full stored-row endpoint also takes five columns. The right panel shows
the first omitted core-response value. Its spectral interpretation
requires the SPD hypotheses of \cref{thm:spectrum}; the plotted
association supplies no universal rank rule or iteration prediction.
\Cref{tab:rank} retains the complete values and residuals.}
\label{fig:interaction-rank-evidence}
\end{figure}
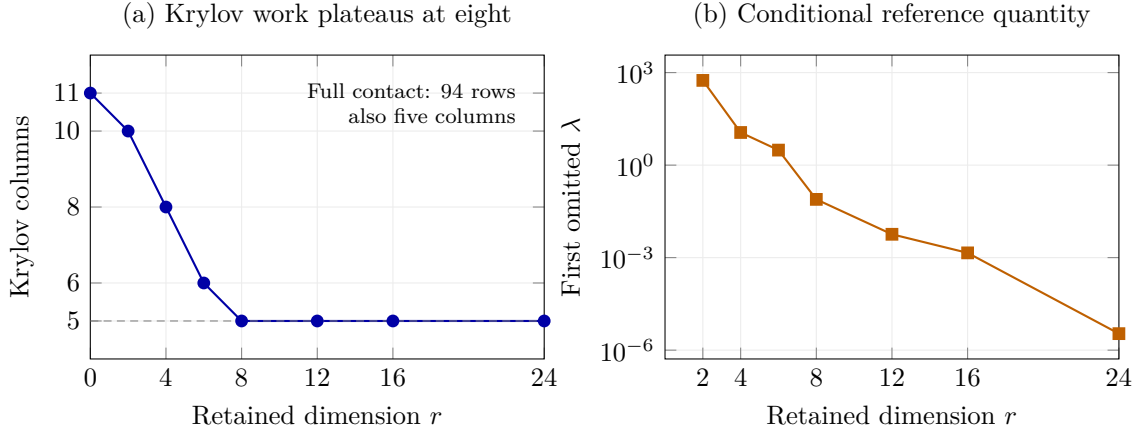

\subsection{From selected interactions to maintained responses}
\label{sec:method-legend}
\label{sec:maintained-economics}

The maintained comparisons use the initial \emph{raw} rank-eight space.
Their reference is either the initial material core or the complete initial
matrix. A held complete factor supplies the same reuse opportunity without
a contact correction. The full-cache alternative retains all supported
coordinate responses and updates the complete supported contact action.
\Cref{tab:method-legend} links these experiment names to the construction.

\begin{table}[H]
\centering\small
\caption{The repeated-solve variants. Every iterative row solves the full
current equation. A reference solve is one application of the stored factor;
response storage excludes factors, sparse operators and Krylov storage.}
\label{tab:method-legend}
\begin{tabularx}{\textwidth}{@{}lXX@{}}
\toprule
Name & Preparation and contact treatment & Inverse action\\
\midrule
Per-step direct & Refactor the current complete matrix; reuse symbolic
analysis where permitted. & Direct solve.\\
Held complete & Initial complete factor; no current contact correction.
& One reference solve.\\
Material, balanced eight & Initial raw eight responses through $M_0$;
update current images relative to zero stored contact.
& One reference solve and the balanced correction.\\
Complete, balanced eight & Initial raw eight responses through $A_0$;
update the signed difference from its initial contact.
& One reference solve and the balanced correction.\\
Complete, cached full & All supported-coordinate responses through $A_0$;
update the full supported contact system.
& One reference solve and a full-cache correction.\\
Implicit full contact & Supported compliance and reduced factor;
omit the full displacement-response cache.
& Two reference solves; a separate realization.\\
\bottomrule
\end{tabularx}
\end{table}

The primary full factor endpoint has $94$ rows, whereas a supported-coordinate
cache uses $60$ columns. They are different full-contact representations.
The changing-equation sequence has its own $61$-coordinate support.
Frozen ordinary correction holds both its space and weights; refreshed raw
correction reacquires them; balanced maintenance retains responses and updates
their current images.

\FloatBarrier
\subsection{What must change when contact changes}
\label{sec:maintained-results}

The first five saved Newton states keep their contact matrices inside the
initial $61$-coordinate support, while the relative Frobenius change from
initial contact reaches $0.9071$.
We compare the initial equation, the largest-change endpoint and the last
equation at true relative residual $10^{-9}$.
Material-reference methods retain $M_0$; complete-reference methods retain
$A_0$ and update its signed contact difference.
All compressed rows use the initial raw rank-eight loads.
Each equation also receives a fixed-seed full-free-coordinate RHS
perturbation, with first use and five subsequent solves.

\begin{table}[H]
\centering\small
\caption{Columns and warm milliseconds on the three current equations,
with their actual RHSs at relative residual tolerance $10^{-9}$. The complete-reference
methods share the same initial complete factor.}
\label{tab:maintained-columns}
\begin{tabular}{@{}lrrr@{}}
\toprule
Reference and correction & Initial & Largest change & Last\\
\midrule
Material, balanced eight & 5 / 37.80 & 5 / 37.98 & 7 / 52.20\\
Complete, balanced eight & 1 / 9.04 & 4 / 29.07 & 5 / 35.80\\
Complete, cached full & 1 / 9.11 & 3 / 22.57 & 4 / 29.32\\
Held complete factor & 1 / 9.04 & 12 / 82.58 & 16 / 110.25\\
\bottomrule
\end{tabular}
\end{table}

The material-reference control isolates the update mechanism.
Frozen ordinary raw eight takes $5/13/20$ columns; refreshed raw and
maintained balanced eight both take $5/5/7$.
The two later balanced updates cost $0.612/0.633$~ms, against
$8.406/8.121$~ms for raw refresh.
Thus the existing responses remain useful when their current images are
updated. Full support also permits inexpensive updates, and is retained
as the comparator.

Complete anchoring further lowers work, although full contact remains faster
on the later equations.
The perturbed RHSs give $1/5/6$ columns for complete balanced eight,
$1/4/5$ for complete cached full, $1/12/17$ for held complete, and
$8/8/8$ for the material reference.
Raw response spaces formed through the two references differ by principal
angles up to $0.0732$ radians; the ideal metric equivalence in
\cref{lem:complete-anchor-metric} has narrower hypotheses.
All $432$ iterative and $36$ direct observations complete, with an
independent check of $78$ first vectors.

Reducing response error alone does not predict a benefit.
At two current cores, old-response defects $0.01937$ and $0.04198$ can be
reduced to about $2\times10^{-11}$ by renewal, yet old and renewed spaces
take the same five and six columns when their actual current images are used.
The construction costs and space-change checks are retained in
\cref{app:reset-diagnostics}.
This observation complements the error identity \eqref{eq:maintained-defect}.

\FloatBarrier
\subsection{Complete repeated quasi-static load studies}
\label{sec:complete-prepared-loads}

Five methods execute the three histories
$(150,165,180)$, $(145,130,160)$ and $(140,155,170)$~Pa on the frictionless
primary problem, restoring the same complete initial state before each
history. One solver instance retains eligible preparation across histories;
every target requires a new solve.
The methods run in two method orders with fresh processes.
All iterative methods use relative linear residual tolerance $10^{-9}$ and the same
contact geometry, pressure potential, collision detection, energy Armijo
and nonlinear stopping criteria.

Every method converges for the same eight of nine targets in both orders.
The held and complete-anchor methods use $58$ Newton iterations in those
targets and one attempt in the unconverged target.
Let $\mathbf r^{\,c}_{ab}\in\R^3$ be the normal-contact residual assembled
on body $a$ from contact with body $b$.
The physical reaction is $\mathbf f_{a\leftarrow b}=-\mathbf r^{\,c}_{ab}$.
We compare Euclidean norms of the concatenated reaction vectors over
all ordered body pairs. Complete balanced eight differs from paired direct by at most
$5.862\times10^{-12}$ in relative displacement and
$3.856\times10^{-10}$ in relative normal reaction; contact-identity checks
and pair counts agree.
These are accepted normal reactions, with zero friction at the stated
coefficient. Full nodal Dirichlet reactions are outside that observation.

\begin{table}[H]
\centering\small
\setlength{\tabcolsep}{4pt}
\caption{Ranges across the forward/reverse method orders. Warm time contains
the seven converged targets after first use. Complete run time includes
model and solver preparation, all nine attempts, reset, output and
cleanup. Response arrays exclude other GPU allocations.}
\label{tab:complete-prepared-loads}
\begin{tabular}{@{}lrrrr@{}}
\toprule
Method & Columns & Warm (s) & Complete (s) & Responses (MB)\\
\midrule
Per-step direct & --- & 107.27--107.30 & 338.20--339.37 & ---\\
Held complete & 1531 & 79.23--79.88 & 323.04--324.09 & ---\\
Material, balanced eight & 671--682 & 75.62--77.77 & 335.11--336.87 & 21.19\\
Complete, balanced eight & 591 & 75.87--76.35 & 320.32--322.11 & 21.19\\
Complete, cached full & 583 & 75.86--76.16 & 320.81--321.65 & 158.90\\
\bottomrule
\end{tabular}
\end{table}

Within each order, complete balanced eight saves $28.82$--$29.30\%$ of
warm converged time against per-step direct and $4.24$--$4.42\%$ against
held complete. The complete run savings are $5.09$--$5.29\%$ and
$0.61$--$0.84\%$, respectively.
Compressed and full-cache time orderings reverse between method orders; they
are effectively tied at this scale, with a $7.5$ response-storage ratio.
Direct performs $59$ numerical factorizations with one symbolic analysis;
each iterative method retains one factor.

Preparation matters to the anchor comparison.
The material-core preparation costs about $34.3$~s and the complete-factor
preparation about $19.8$~s. Their common factor-construction times from assembled CPU matrices are much closer,
so this difference is attributed to the preparation procedures.
The shared-machine runs give paired observations and descriptive ranges,
with no reliable population confidence interval for sub-percent gains.

All methods stop at the $130$-to-$160$~Pa reversal during construction of a
new contact-history reference. The returned state has not converged.
The reverse material-reference run also takes one additional iteration
after a near-zero Armijo margin changes sign.
\Cref{app:reset-diagnostics} retains these geometric and arithmetic
observations. They qualify completion and repeatability; they do not
change the successful complete-equation linear residuals.

\FloatBarrier
\subsection{Reuse across further reset episodes}
\label{sec:reset-episode-results}

The extended program has twelve episodes of four pressure targets.
Its first target is $150$~Pa; other starts are uniform in $[130,170]$~Pa,
increments are uniform in $[-20,20]$~Pa, and pressures stay in
$[120,190]$~Pa. The fixed table, generated with seed $2026090811$ before
execution, is shared by all methods.
An unconverged target ends its history; independent later resets remain eligible.
The physical state is restored while the numerical factor and response
vectors are retained.

Seven runs return $45$ targets and converge for the same $43$.
The reverse compressed run terminates during the energy-based Armijo line search at
its second target, after one converged target.
Its shortened time is excluded from equal-output savings.

\begin{table}[H]
\centering\small
\setlength{\tabcolsep}{4pt}
\caption{Forward/reverse reset-program observations. Complete time includes
model and solver preparation, unsuccessful work, resets, writing results and releasing solver resources.
Warm time sums converged targets after first use. The starred run
terminates exceptionally at its second target and yields no warm-time
or whole-program speedup. Each complete run converges for the same $43$
targets and returns two unconverged targets.}
\label{tab:reset-episodes}
\begin{tabular}{@{}lrrr@{}}
\toprule
Method & Complete (s) & Warm (s) & Converged\\
\midrule
Per-step direct & $883.36/882.59$ & $633.49/633.14$ & $43/43$\\
Held complete & $737.59/736.54$ & $473.11/471.83$ & $43/43$\\
Complete, balanced eight & $716.91/254.77^*$ & $450.14/\text{---}$ & $43/1^*$\\
Complete, cached full & $724.47/721.32$ & $454.11/452.91$ & $43/43$\\
\bottomrule
\end{tabular}
\end{table}

The complete forward compressed run saves $18.84\%$ of total
time and $28.94\%$ of warm converged time against direct; the corresponding
held-factor reductions are $2.80\%$ and $4.86\%$.
Full cache completes both orders with $17.99$--$18.27\%$ lower total time
than direct.
Each iterative run retains one factor. Compression acquires eight
responses with no reacquisition; full cache acquires sixty.
All $85$ archived resets reproduce the initial physical-state bytes.

\begin{figure}[H]
\centering
\begin{tikzpicture}
\begin{axis}[
  at={(0,0)},anchor=south west,
  width=0.68\linewidth,height=6.5cm,
  ybar stacked,bar width=20pt,ymin=0,ymax=1000,
  xmin=-0.65,xmax=3.65,xtick={0,1,2,3},
  xticklabels={Current\\direct,Held\\factor,Balanced\\eight,Cached\\full},
  xticklabel style={align=center,font=\small},
  ylabel={Elapsed time (s)},label style={font=\small},
  tick label style={font=\small},
  title={(a) Forward program: 43 converged targets},
  title style={font=\small},
  ymajorgrids=true,grid style={gray!18},
  legend style={at={(0.5,-0.24)},anchor=north,font=\scriptsize,
    legend columns=1,draw=none,fill=none}]
\addplot[draw=blue!65!black,fill=blue!40] coordinates {
  (0,633.49) (1,473.11) (2,450.14) (3,454.11)};
\addlegendentry{Warm converged work}
\addplot[draw=gray!65,fill=gray!30] coordinates {
  (0,249.87) (1,264.48) (2,266.77) (3,270.36)};
\addlegendentry{Other included work}
\node[above,font=\scriptsize] at (axis cs:0,883.36) {883.36};
\node[above,font=\scriptsize] at (axis cs:1,737.59) {737.59};
\node[above,font=\scriptsize] at (axis cs:2,716.91) {716.91};
\node[above,font=\scriptsize] at (axis cs:3,724.47) {724.47};
\end{axis}
\begin{axis}[
  at={(10.3cm,0)},anchor=south west,
  width=0.30\linewidth,height=6.5cm,
  ybar,bar width=20pt,ymin=0,ymax=9,
  xmin=-0.65,xmax=1.65,xtick={0,1},
  xticklabels={Balanced\\eight,Cached\\full},
  xticklabel style={align=center,font=\small},
  ylabel={Relative response storage},label style={font=\small},
  ytick={0,2,4,6,8},tick label style={font=\small},
  title={(b) Response storage only},title style={font=\small},
  ymajorgrids=true,grid style={gray!18}]
\addplot[draw=blue!65!black,fill=blue!40] coordinates {(0,1) (1,7.5)};
\node[above,font=\small] at (axis cs:0,1) {1};
\node[above,font=\small] at (axis cs:1,7.5) {7.5};
\end{axis}
\end{tikzpicture}
\caption{Time and response storage in the extended reset program.
(a) The four completed forward runs converge for the same 43 targets;
bar totals reproduce \cref{tab:reset-episodes}. The stack separates warm
converged work from all other included work. (b) Response storage normalized
to the compressed cache. The reverse compressed run stops after
one converged target and is excluded from this equal-output comparison;
that termination remains recorded in \cref{tab:reset-episodes}. These shared-machine
observations give no population timing confidence interval or total-VRAM
reduction.}
\label{fig:reset-program-economics}
\end{figure}
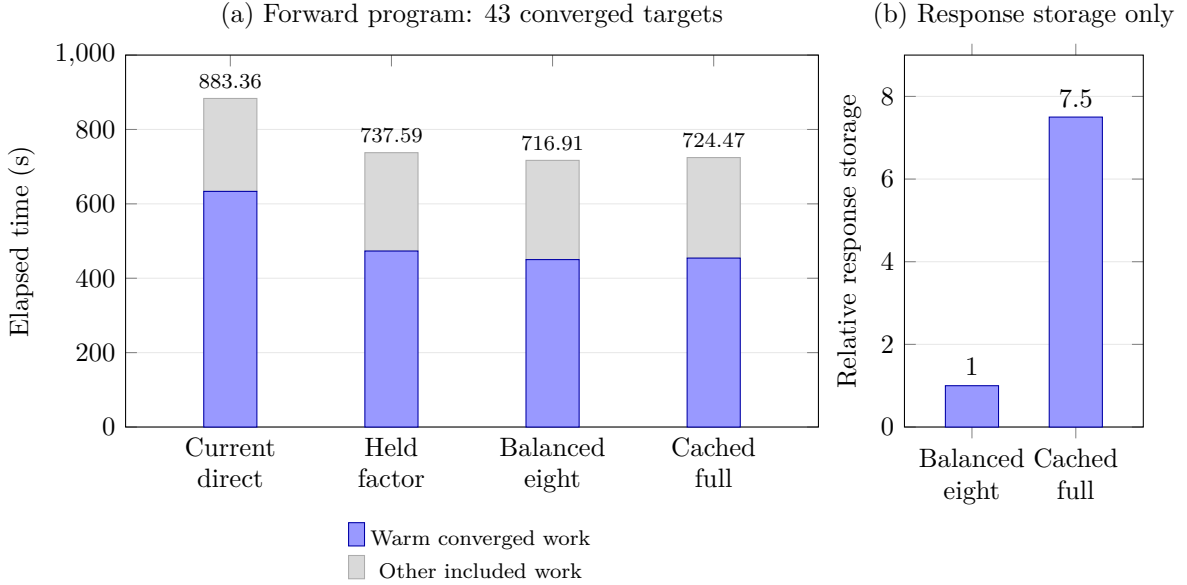

The longer study has a wider accuracy envelope than the short paired studies:
the largest forward compressed/direct relative displacement and normal-reaction
differences are $1.770\times10^{-6}$ and $2.322\times10^{-5}$.
Direct repeats exhibit essentially the same endpoint variation.
The recorded terminal step fractions explain that variation at the stopping
scale. A separate $60$-digit evaluation of a captured state changes the sign
of a cup-energy decrement, identifying constitutive cancellation.
The original failed-trial vectors were not saved, so that calculation does
not reconstruct the exception.
The full diagnostic values in \cref{app:reset-diagnostics} preserve this
distinction. All reported timings use the original formula; the failed
repeat remains a nonlinear robustness limitation.

\subsection{Dependence on contact, loading and nonlinear history}
\label{sec:first-contact-comparison}
\label{sec:controls}
\label{sec:nonlinear-consequences}

The retained rank is problem dependent.
The first-contact equation has $856{,}206$ free coordinates, eight normal
pairs and a $24$-row normal-contact factor; friction remains in the
full-equation remainder. Eight raw or metric directions take
fifteen columns and about $296$~ms against eighteen to nineteen columns and
about $354$~ms for core-only. The corresponding core/operator setup costs
$60.167/50.744$~s.
On the primary equation at forcing $10^{-3}$, metric eight takes one column
against six for core-only; at forcing $0.1$, both take one.
\Cref{sec:forcing-observations} gives the full accuracy comparison.
Its direction differences also show why a residual tolerance alone is not a
forward-error or nonlinear-progress guarantee.

A separately staged high-contact equation gives a direct adverse case.
At forcing $0.1$, core-only takes three columns, while raw and metric rank
eight reach the $256$-column limit at residuals $0.188$ and $0.244$.
Rank $128$ and full contact both take one column.
The independent action checks in \cref{sec:high-contact-qualification}
reproduce the regression; \cref{rem:partial-cluster} explains how a partial
spectral correction can worsen the first residual even in exact SPD algebra.

On the difficult late equation, current-core raw and metric rank $128$
take $77$ and $76$ columns, while full supported contact takes one.
The metric first-omitted value is about four times smaller than the raw
omission, with almost unchanged work.
The corresponding implicit full-contact representation stores small
supported data and pays two core applications; it is not the full displacement
cache used in the primary maintained study.
\Cref{sec:late-contact-representation} retains the matched construction costs
and representation sizes.

Correction strength is a separate control on a chosen space.
The positive spectral family and supported first-direction optimization in
\cref{app:rhs-strength} preserve a mathematically specified inverse family.
On one late-state trial, the selected direction lowers the returned residual
from $168.672$ to $48.407$, against $70.814$ for core-only.
Its small first residual does not ensure a good full solve: held-out
prepared loads reach the $480$-column limit, and at forcing $10^{-3}$ the
selected direction has a relative displacement difference $0.998$ from direct.
The complete queries and one-step comparisons are in
\cref{sec:prepared-load-coverage,sec:short-direction-consequence}.
This strength variant is not used in the principal raw-eight maintained
experiments.

Complete frictional histories reinforce the acquisition-cost limitation.
The selected-strength method uses $122$ solve columns against core-only's
$5231$, but acquires $14431$ setup responses and takes $1886.882$~s against
$1331.118$~s. Full supported contact takes $4032.666$~s; tighter ordinary
raw correction converges in $1318.959$~s, and direct in $901.397$~s.
\Cref{tab:trajectory-summary} retains all completed and iteration limit outcomes.
A good first direction and a small solve-phase count therefore do not
establish economical complete execution.

The final friction forces also depend on the contact-entry history.
The complete comparison in \cref{sec:history-consequence} reproduces the
frozen force map and isolates differing per-pair reference coordinates.
Strict final residuals do not identify one common friction history.
Together with the shared geometric-query failures and exceptional compressed
repeat above, these observations bound the physical and numerical regime
of the reported savings. No dynamic, learner-throughput or general
Krylov-recycling performance claim is inferred.

The frozen volumetric scaling control preserves eleven core-only versus five
rank-eight columns for scale factors $0.25$ through $4$; its largest positive
departure from predicted monotonicity is $2.62\times10^{-9}$.
It tests \cref{eq:pressure-monotone} on fixed algebra; re-equilibrated
material behaviour is a different question. Empty contact recovers the core-only action.

\section{Discussion}
\label{sec:discussion}

The construction connects a mechanical definition of interaction relevance
to a small retained response space and its current-equation use.
The SPD reference results explain what the metric truncation omits and why
that choice is optimal for its stated worst-energy criterion.
The raw proxy tests how much of this useful space can be acquired cheaply.
Supported image updates then preserve the selected responses' current action
without reacquiring them at every equation.

The experiments separate the value of augmentation, compression and reuse.
On the primary equation, the largest work reduction comes from adding useful
contact information; compression retains full-contact work with lower action
cost and response storage. On changing equations, current images matter more
than simply retaining the initial correction weights.
The complete anchored studies establish a useful warm benefit and a smaller
preparation-inclusive benefit against direct refactorization.
Held-factor and full-cache comparisons locate the additional contribution:
the correction improves on held-factor reuse, while matching full-cache time
with much less response storage in the demonstrated regime.

The proofs also identify what must remain controlled. Reference positivity,
the omitted complement, full-space operator drift and inverse-action error
govern the linear estimates. Tangent consistency, smoothness and accepted
step fractions enter the local nonlinear result.
The physical experiments establish their actual residuals and numerical convergence;
they do not establish every uniform hypothesis.
A fixed support alone guarantees neither space quality nor nonlinear
robustness. The adverse rank, held-out and friction-history observations
are consequences to understand within that scope.

\section{Conclusion}

TOGEARI selects contact interactions through their mechanical responses,
constructs a compressed inverse correction and maintains it through supported
signed contact changes while solving the complete Newton equation.
Its reference and current-equation analysis links truncation, retained images,
operator drift and local nonlinear convergence under explicit hypotheses.
The cheaper raw selector and matched comparisons make the acquisition and
maintenance costs part of the same method.

The demonstrated regime combines useful warm and complete-work savings with
near-full-cache time and $13.3\%$ of its response storage.
Large-contact, held-out and nonlinear-history results delimit that regime.
The analysis and comparisons explain when compressed contact responses
remain useful through a complete nonlinear calculation, and what their
construction and maintenance cost in the demonstrated regimes.

\clearpage
\section*{CRediT author statement}
Yanlin Liu originated the central theoretical and algorithmic ideas and
mathematical formulations of TOGEARI and led the theoretical development.

\noindent\textbf{Yanlin Liu:} Conceptualization (central theoretical and
algorithmic ideas); Methodology (mathematical formulation and algorithm design);
Formal analysis (theoretical analysis and derivations);
Software (research-specific development, framework extensions and performance
optimization); Investigation; Validation; Writing---original draft;
Writing---review and editing.

\noindent\textbf{Chao Huang:} Software (development of the foundational
simulation framework); Conceptualization (introduction of the underlying
numerical problem and the near-incompressibility regime).

\noindent\textbf{Kaixiang Yao:} Conceptualization (identification of the
warm-execution-time acceleration requirement in repeated robotics simulations);
Writing---review and editing.

\noindent\textbf{Yao Shen:} Supervision; Resources (provision of computational
resources); Conceptualization (broad problem introduction and research context).

\section*{Data and code availability}
Retained numerical observations, solver configurations, and controlled
analysis inputs are available from the authors upon reasonable request.
The private simulation platform and its production solvers are not publicly
distributed. A public release of selected analysis and reproduction
materials may accompany a subsequent version.

\section*{Declaration of generative AI use}
OpenAI Codex was used to assist with experimental coding, preliminary
mathematical review, and preliminary drafting. Anthropic Claude was
used for draft reviewing and editing. Yanlin Liu provided the central
theoretical ideas and initial mathematical formulations.
Responsibility for the claims, citations, computational results,
and final manuscript rests with the named authors.

\clearpage
\appendix
\setcounter{table}{0}
\renewcommand{\thetable}{\Alph{section}.\arabic{table}}
\renewcommand{\theHtable}{appendix.\Alph{section}.\arabic{table}}

\section{Complete retained-dimension sweep}
\label{app:rank-sweep}

\begin{table}[H]
\centering\small
\caption{Complete metric retained-dimension sweep on the primary equation.
The entries use the fixed-equation reference implementation. The final
column is the first omitted core-response value, whose condition-number
interpretation is conditional on the SPD hypothesis.}
\label{tab:rank}
\begin{tabular}{@{}rrrr@{}}
\toprule
Retained dimension $r$ & Krylov columns & Relative residual & First omitted $\lambda$\\
\midrule
0 & 11 & $8.861\times10^{-10}$ & ---\\
2 & 10 & $1.260\times10^{-10}$ & $5.592\times10^2$\\
4 & 8 & $2.262\times10^{-10}$ & $1.144\times10^1$\\
6 & 6 & $6.411\times10^{-10}$ & $3.071$\\
8 & 5 & $2.902\times10^{-10}$ & $7.746\times10^{-2}$\\
12 & 5 & $2.853\times10^{-10}$ & $5.716\times10^{-3}$\\
16 & 5 & $2.105\times10^{-10}$ & $1.426\times10^{-3}$\\
24 & 5 & $2.067\times10^{-10}$ & $3.429\times10^{-6}$\\
94 & 5 & $2.067\times10^{-10}$ & $0$\\
\bottomrule
\end{tabular}
\end{table}

\section{Fixed-equation sensitivity}

The following comparisons hold their original operator and RHS fixed.
They supply accuracy- and rank-dependent observations; they do not
represent additional nonlinear histories.

\subsection{Dependence on the relative forcing level}
\label{sec:forcing-observations}

The two frozen equations use $\rho(x)=\|b-Ax\|_2/\|b\|_2\le\eta$,
zero absolute tolerance and
$\eta\in\{0.1,0.03,0.01,10^{-3},10^{-6},10^{-9}\}$.
Core-only, raw eight, metric eight and full metric correction share the
GPU recurrence, complete equation, precision and core factor.
Each order and its reverse use a first solve and three warm solves,
giving $48$ configurations and $192$ solves per equation.
All report successful termination and meet the independent complete-residual
criterion. The same numerical factorization is reused.

\begin{table}[H]
\centering
\small
\caption{Forcing-level comparison on two fixed complete equations.
Each time range spans the two reversed-order warm medians;
each median uses three solves and includes terminal solution
transfer. These ranges are within-process observations.}
\label{tab:forcing-levels}
\begin{tabular}{@{}lcrrrr@{}}
\toprule
Equation & $\eta$ & Core columns & Metric-8 columns
& Core (ms) & Metric-8 (ms)\\
\midrule
Primary & $0.1$ & 1 & 1 & 9.064--9.782 & 9.151--9.179\\
& $0.03$ & 3 & 1 & 23.030--24.748 & 9.136--9.152\\
& $0.01$ & 3 & 1 & 23.068--24.287 & 9.157--9.158\\
& $10^{-3}$ & 6 & 1 & 44.210--44.235 & 9.157--9.184\\
& $10^{-6}$ & 10 & 2 & 72.780--72.784 & 16.193--16.202\\
& $10^{-9}$ & 11 & 5 & 80.019--80.164 & 37.511--37.523\\
\midrule
First contact & $0.1$ & 6 & 4 & 117.785--117.989 & 80.291--80.439\\
& $0.03$ & 7 & 5 & 137.044--137.088 & 99.388--99.484\\
& $0.01$ & 7 & 5 & 136.817--137.135 & 99.497--99.576\\
& $10^{-3}$ & 8 & 5 & 156.207--156.245 & 99.469--99.483\\
& $10^{-6}$ & 12 & 9 & 234.267--234.310 & 176.917--177.028\\
& $10^{-9}$ & 18--19 & 15 & 354.101--375.637 & 296.125--296.271\\
\bottomrule
\end{tabular}
\end{table}

One strict first-contact core solve uses nineteen columns; the others use
eighteen. Raw eight and full correction match the metric-eight counts
at every level.
At $\eta=10^{-3}$, full-correction warm medians are
$9.605$--$9.638$~ms and $100.700$--$100.754$~ms on the two equations.
The main gain measures augmentation; compression adds a smaller saving.
Core/operator setup costs $22.359/18.117$~s on the primary equation and
$60.198/50.730$~s at first contact.
Metric selector setup is $0.109/0.0803$~s; warm raw setup is
$0.00171/0.00160$~s. Rebuilding a metric correction can exceed a single
warm solve saving.

Let $x_{\rm ref}$ be the first core-only solution at $\eta=10^{-9}$,
whose relative residuals are $8.847\times10^{-10}$ and $4.812\times10^{-10}$.
Define $d(x)=\|x-x_{\rm ref}\|_2/\|x_{\rm ref}\|_2$ as a numerical
direction comparison, without a certified forward-error interpretation.
On the primary equation at $\eta=0.1,0.03,10^{-3}$, core-only gives
$d\simeq1.1225,0.2808,0.6737$, while metric eight gives about $0.0753$.
At $10^{-6}$ the values are $4.42\times10^{-4}$ and $1.48\times10^{-2}$.
At first contact with $\eta=0.1$, they are $0.0825$ and $0.0982$.

For an exact solution $x_*$, $x-x_*=A^{-1}(Ax-b)$.
The residual therefore requires conditioning information to bound direction
error, and sampled direction differences need not decrease monotonically.
The nonlinear consequences are measured separately under the same physical
trial rules.

\subsection{A high-contact retained-dimension comparison}
\label{sec:high-contact-qualification}

This frozen equation comes from a separately staged history of the
four-finger/cup problem. It has $865{,}572$ stored coordinates, including
$856{,}206$ free coordinates, at pressure $750$~Pa and radial closure
$1.25$~mm. Friction coefficient $0.3$ and regularization length
$10^{-3}$~mm are retained. This state is used only for its fixed-equation
comparison; it is separate from the complete histories of
\cref{sec:trajectory-summary}. Write $n_{\rm e}=865{,}572$ for the
embedded dimension, with prescribed identity rows included.

Let $C$ denote the sum of the normal and lagged-friction contact
tangents on that equation. The captured local roots have $5320$ rows and
reconstruct $C$ with relative Frobenius discrepancy
$4.73166\times10^{-13}$. The original assembled $C$ is supported on
$p=2091$ coordinates. Let $E\in\R^{p\times n_{\rm e}}$ restrict a full vector to
these coordinates and define $C_s=ECE^T$. Then $C=E^TC_sE$.
The reference core $M$ is assembled from the same finite-element material assembly
as the full matrix $A$. The relative asymmetries of $A$ and $M$ are
$1.71447\times10^{-16}$ and $1.72354\times10^{-16}$.
cuDSS in SPD mode reports successful factorization and solves
the original $A$ and $M$ for the actual RHS to relative residuals
$1.55405\times10^{-13}$ and $4.32870\times10^{-11}$.
These are numerical factorization checks; the exact SPD
theorems retain their mathematical hypotheses.

An explicit dense transpose of the captured contact root would occupy
$36.84$~GB. The support representation avoids that allocation. Define the
support compliance $W_s=EM^{-1}E^T$. Eight-column factor applications form
$W_s$ in $262$ blocks, including five padded zero columns. Its relative
asymmetry is $2.86447\times10^{-15}$. Define the symmetric supported
contact array $\widehat C_s=(C_s+C_s^T)/2$. Raw selection uses its leading
eigenvectors. For weighted selection, write
$W_s=L_sL_s^T$ after numerical symmetrization and let
$\theta_i$ and $v_i$ be the descending eigenvalues and orthonormal
eigenvectors of $L_s^T\widehat C_sL_s$. The retained supported root has rows
$\sqrt{\theta_i}\,v_i^TL_s^{-1}$. The equivalence follows from the
two Gram products in a singular-value decomposition \citep{stewart1990matrix}. In exact arithmetic,
assume $M\succ0$ and $C_s\succeq0$, and set $T=C_s^{1/2}$.
The interaction Gram matrix for the root $TE$ is
$(TL_s)(TL_s)^T$, while its right Gram matrix is $L_s^TC_sL_s$.
For a positive eigenvalue $\theta_i$, the corresponding normalized
left vector is $q_i=TL_sv_i/\sqrt{\theta_i}$. The eigenvector equation
then gives $q_i^TT=\sqrt{\theta_i}\,v_i^TL_s^{-1}$, which is the
retained row stated above. The symmetrization of $C_s$ changes
it by relative Frobenius norm $9.66548\times10^{-17}$; every Krylov action
continues to use the original complete $A$.

For a supported retained root $R\in\R^{r\times p}$, cache
$X=M^{-1}E^TR^T$ and apply the Woodbury formula to
$P=M+E^TR^TRE$. This uses one core solve per action. The full-contact
control avoids caching all $p$ response columns. For an input vector $v$,
it computes
\[
q=M^{-1}v,\qquad
h=(I_p+\widehat C_sW_s)^{-1}\widehat C_sEq,\qquad
z=q-M^{-1}E^Th.
\]
For exact $W_s=EM^{-1}E^T$, the identity
$(M+E^T\widehat C_sE)^{-1}v=z$ follows from the Woodbury formula whenever
the displayed inverses exist. The numerical control pays for two core applications.
This economical compliance baseline is consistent with established
contact-space constructions \citep{zeng2022contact}.

\begin{table}[H]
\centering
\caption{Fixed-equation comparison at iteration $30$ and relative forcing $0.1$.
All methods use the same original equation, restart width $40$, and
$256$-column limit. The two rank-eight rows reach that iteration limit.
The final column counts core applications during Krylov iteration.}
\label{tab:high-contact-ranks}
\begin{tabular}{lrrr}
\toprule
Correction & Columns & True relative residual & Core applications\\
\midrule
Core only & 3 & $8.49406\times10^{-2}$ & 3\\
Raw 8 & 256 & $1.88026\times10^{-1}$ & 256\\
Raw 32 & 29 & $5.85872\times10^{-2}$ & 29\\
Raw 128 & 1 & $4.16824\times10^{-2}$ & 1\\
Weighted 8 & 256 & $2.44281\times10^{-1}$ & 256\\
Weighted 32 & 29 & $5.47537\times10^{-2}$ & 29\\
Weighted 128 & 1 & $9.10877\times10^{-2}$ & 1\\
Full supported (implicit) & 1 & $4.61582\times10^{-4}$ & 2\\
\bottomrule
\end{tabular}
\end{table}

Independent action checks reproduce the adverse rank-eight behavior.
The GPU-assembled response loads agree exactly with their independently constructed CPU counterparts; the largest
original-$M$ response residual is $2.55980\times10^{-9}$ and
$\|Pz-b\|_2/\|b\|_2\le2.41181\times10^{-10}$.
Raw eight has first residual $0.9683166$ under both $A$ and $M+C$,
against core-only's $0.1941799$; weighted eight behaves similarly.
Direct factorization of the same raw-eight preconditioner agrees with
Woodbury to relative direction difference $1.70707\times10^{-9}$ and first
residual difference $2.18\times10^{-13}$.
Removing omitted pressure terms leaves the initial regression essentially
unchanged. \Cref{rem:partial-cluster} supplies a compatible first-step
mechanism; restart and Euclidean geometry also enter the full history.

The implicit full-contact action has defining-equation residuals
$2.31135\times10^{-6}$ for the actual RHS and $8.54661\times10^{-6}$ for
a deterministic sine vector. They are below the tested forcing level but
matter at tighter accuracy.
This comparison uses CPU sparse matrix products and Krylov bases and GPU V100
core solves. Core preparation and factorization costs $42.35$~s and compliance
$30.84$~s. Cached ranks $8,32,128$ occupy $55.4,221.6,886.3$~MB;
the implicit full-contact control has no full-dimensional response cache.
The complete GPU trajectories are a separate comparison in
\cref{sec:trajectory-summary}.

\section{Core inverse regimes beyond the exact construction}
\label{app:core-regimes}
Positive definiteness is the certification hypothesis for the metric theory.
The right-preconditioning action has weaker algebraic requirements.  For a
fixed nonsingular $M$ and retained row matrix $R_r$, \eqref{eq:woodbury} is
well-defined whenever $S_r=I_r+R_rM^{-1}R_r^T$ is nonsingular.  The metric theorem
assumes the stronger conditions $M=M^T\succ0$ and an exact action $M^{-1}$.  A
finite-precision direct factorization approximates this action; the exact
theorem describes the corresponding ideal preconditioner, and the effect of
solve error requires a separate perturbation analysis.

For a fixed approximate SPD core $\widetilde M=\widetilde M^T\succ0$ used
consistently in selection, response construction, and preconditioner
application, define
\begin{equation}
 \widetilde G=U\widetilde M^{-1}U^T,
 \quad C_r=U^T\Pi_rU,
 \quad E_r=U^T(I_m-\Pi_r)U.
\label{eq:approximate-split}
\end{equation}
Here $\Pi_r\in\R^{m\times m}$ is the rank-$r$ orthogonal factor-space
projector used by the approximate construction.
Then $C_r,E_r\succeq0$, $U^TU=C_r+E_r$, and $H=M+C_r+E_r$. Suppose
\begin{equation}
 0<c_1\le c_2,
 \quad c_1\widetilde M\preceq M\preceq c_2\widetilde M,
 \quad E_r\preceq\delta\widetilde M,
 \quad \delta\ge0.
\label{eq:equiv}
\end{equation}
For $\widetilde P_r=\widetilde M+C_r$,
\begin{equation}
 \min\{c_1,1\}\widetilde P_r
 \preceq H\preceq
 \max\{c_2+\delta,1\}\widetilde P_r,
\end{equation}
and therefore
\begin{equation}
 \kappa_2(\widetilde P_r^{-1/2}H\widetilde P_r^{-1/2})
 \le
 \frac{\max\{c_2+\delta,1\}}{\min\{c_1,1\}}.
\label{eq:inexact-bound}
\end{equation}
To see the lower inequality, use
$M+C_r+E_r\succeq c_1\widetilde M+C_r$ and the nonnegativity of
$\widetilde M$ and $C_r$. The upper inequality follows from
$M+C_r+E_r\preceq(c_2+\delta)\widetilde M+C_r$ in the same way.
Congruence by $\widetilde P_r^{-1/2}$ bounds every eigenvalue between the
two displayed constants and gives \eqref{eq:inexact-bound}.
This conditional bound concerns the SPD pair.  The exact spectrum in
\cref{thm:spectrum} and the complete matrix $A$ require their respective
analyses.

For a fixed symmetric indefinite $M$, the matrix $G_M$ remains symmetric and
has signed response values.  The descending ordering used here is empirical;
the SPD energy and optimality interpretation belongs to the positive-definite
regime.  A nonsymmetric $M$ or effective inverse action generally produces a
nonsymmetric $G_M$, placing core-response selection outside the present
analysis.  A variable or tolerance-varying inner solver defines a distinct
preconditioning action.  Flexible GMRES can use these actions, with the
complete-equation residual providing the acceptance test.

\section{Supporting contact and nonlinear comparisons}
\label{app:additional-comparisons}

These comparisons qualify the selected-space and maintenance results.
The difficult late equation, the separate high-contact equation in
\cref{sec:high-contact-qualification}, and the complete nonlinear histories
are distinct observations. The correction-strength construction is proved
in \cref{app:rhs-strength}; it is an additional control on a selected space.

\subsection{The difficult late equation}
\label{sec:late-contact-representation}

This captured equation has $865{,}572$ stored coordinates, a $9913$-row
contact root and $2272$ supported coordinates.
Write its matrix, RHS and contact tangent as $A_\star,b_\star,C_\star$,
and its current and initial complete noncontact cores as
$\overline M_\star,\overline M_0$.
The contact reconstruction discrepancy is $8.64036\times10^{-14}$;
the free-coordinate discrepancy from the symmetric current-core/contact
reference is $8.40596\times10^{-14}$.

The first three rows of \cref{tab:late-contact-representation} share a
current-core factor and right-FGMRES on CPU with two orthogonalization passes,
forcing $10^{-5}$, restart $80$ and a $480$-column limit.
The initial-core row uses the same equation and recurrence in a separate
execution. Full support uses \eqref{eq:implicit-full-contact}.

\begin{table}[H]
\centering\small
\caption{Contact representations on the difficult late equation.
Solve-phase core applications exclude response acquisition.}
\label{tab:late-contact-representation}
\begin{tabular}{@{}llrrr@{}}
\toprule
Core & Correction & Columns & Core applications & True relative residual\\
\midrule
Current & Raw 128 & 77 & 77 & $9.57235\times10^{-6}$\\
Current & Core-conditioned 128 & 76 & 76 & $9.56394\times10^{-6}$\\
Current & Full supported & 1 & 2 & $1.40633\times10^{-8}$\\
Initial & Full supported & 10 & 20 & $7.87056\times10^{-6}$\\
\bottomrule
\end{tabular}
\end{table}

The metric first-omitted value is $2900.992586$, against weighted raw
omission $11862.089566$. Their fourfold difference produces almost the
same work. The raw omission also has a small negative eigenvalue estimate
$-2.29104\times10^{-9}$.
A compatible core-only run reaches the limit at residual
$1.84126\times10^{-4}$, so the selected contact content still supplies a
substantial augmentation benefit.

Current-core setup costs $34.002$~s.
Full compliance requires $2272$ responses and $32.510$~s;
metric selection adds $1.792$~s.
Raw/metric retained responses cost $3.502/3.378$~s for $128$ RHSs,
followed by $52.619/51.853$~s solves with Krylov iteration on CPU.
Full contact has $0.443$~s reduced setup and a $1.252$~s solve.
The initial-core full route costs $34.266$~s for its core,
$36.657$~s for support and $5.354$~s for the solve.
These fixed-equation costs do not establish cross-Newton amortization.
Each cached rank-$128$ response matrix occupies $886.346$~MB; the
two-solve full action retains $82.601$~MB of supported data.

Established spectral normalization changes the weight of the same metric
space. At forcing $10^{-5}$, ordinary/normalized/full corrections take
$76/79/2$ columns. Their prefixes reach $10^{-3}$ after $25/3/1$ columns,
with normalized and ordinary first residuals $0.00306340$ and $0.831655$.
Prefix times are unrecorded. This motivates the direction-specific control
below while preserving the adverse strict-solve ordering.

\FloatBarrier
\subsection{Prepared queries and held-out load coverage}
\label{sec:prepared-load-coverage}

Eight loads share one fixed physical equation: its original RHS, five
perturbations in the retained load span, one supported perturbation
orthogonal to that span and one full-free-coordinate perturbation.
Every perturbation has norm $0.1\|b\|_2$.
Completed configurations have one first solve and five warm repetitions;
a $480$-column limit observation is retained once.
A separate direct process reuses its complete factor.

The current-RHS strength selector is the family in \cref{app:rhs-strength},
which minimizes a first residual on a chosen space.
Cached supported residual images reduce the six related selection queries
from $24.1$--$25.0$~ms to $11.4$--$11.9$~ms, at additional preparation
$6.8$~ms and $2.33$~MB.
The following complete query-and-solve comparison includes selection.

\begin{table}[H]
\centering\small
\setlength{\tabcolsep}{5pt}
\caption{Warm GPU work on the difficult fixed equation at relative
residual tolerance $10^{-3}$. Times are medians of five subsequent solves, including
any current-RHS selection. Held-out entries are Krylov columns;
``limit'' means the $480$-column observation is incomplete. Direct
returns its more accurate numerical factor solution.}
\label{tab:prepared-load-coverage}
\begin{tabular}{@{}lrrrr@{}}
\toprule
Method & Original columns & Original time (s) & Supported held-out & Free held-out\\
\midrule
Ordinary explicit raw 128 & 27 & 0.600 & 119 & 111\\
Selected strength, implicit & 4 & 0.213 & limit & limit\\
Full supported, implicit & 1 & 0.044 & 1 & 1\\
Reused complete direct & --- & 0.0215 & --- & ---\\
\bottomrule
\end{tabular}
\end{table}

The five related perturbations preserve the original load's qualitative
ordering. At forcing $10^{-5}$, original ordinary/selected work is $77/93$
columns; selected held-outs remain incomplete while ordinary completes in
$209/202$ columns.
At $10^{-3}$, the original selected/ordinary/full displacement differences
from direct are $(0.998,0.178,9.17\times10^{-5})$, respectively; at $10^{-5}$ the
selected/ordinary values fall to $(0.0142,0.00948)$.
Equal residual tolerances therefore do not imply equal field accuracy, and a
first-direction objective supplies no held-out or multi-column guarantee.

Core preparation costs $51.65$~s plus $47.72$~s for one-time preparation of complete-matrix products. Selected acquisition and full compliance cost $1.75$ and
$27.15$~s; direct preparation costs $52.34$~s.
The candidates have greater attributed setup and warm costs than direct
reuse on this fixed equation. The primary equation has the same ordering:
ordinary rank eight is already effective.
Changing-equation maintenance in \cref{sec:maintained-results} addresses
the different situation in which the old direct factor no longer solves
the current matrix exactly.

\FloatBarrier
\subsection{A supported first direction and its nonlinear consequence}
\label{sec:short-direction-consequence}

The strength parameter $\theta\ge0$ in \cref{app:retained-gram-strength}
weakens retained rows through a positive-part spectral map.
Crossing raw and metric rank-$128$ pools with high strength
$2900.9925863$ gives returned residuals $46.991/47.153$; low strength
$141.5124915$ gives $199.093/154.694$.
The high value uses the full metric spectrum and defines the oracle
comparator. \Cref{prop:rhs-strength-selection} instead selects strength from
the current RHS and supported actions before seeing that oracle.

Each method receives one column of the same host FGMRES recurrence on
the original late equation. Its saved direction enters one Newton trial
from identical displacement, contact history, lagged friction and controls.
Collision detection, energy Armijo and post-trial evaluation are unchanged.
Let $F^+$ be the returned residual, $\alpha$ the accepted fraction and
$\Delta\Phi$ the frozen-merit decrease; the initial residual is $168.6720$.

\begin{table}[H]
\centering
\small
\caption{One-column directions and their matched nonlinear consequence.
The current-RHS parameter is selected without the full spectrum. All five trials satisfy the geometric checks and use the stated one-column methods. The ordinary raw trial does not satisfy friction consistency;
the other four do.}
\label{tab:rhs-strength-consequence}
\begin{tabular}{@{}lrrrr@{}}
\toprule
Method & Linear relative residual & $\alpha$ & $\norm{F^+}_2$ & $\Delta\Phi$\\
\midrule
Current-RHS raw & $0.002815$ & $1$ & $48.407$ & $0.180558$\\
Ordinary raw & $0.925255$ & $0.007166$ & $173.520$ & $0.005302$\\
Core only & $0.325763$ & $1$ & $70.814$ & $0.119164$\\
Oracle-strength raw & $0.003231$ & $1$ & $46.991$ & $0.179623$\\
Full contact & $0.000134$ & $1$ & $137.576$ & $0.093760$\\
\bottomrule
\end{tabular}
\end{table}

The selected value $1015.5263$ leaves $125$ positive modes.
It gives the largest merit decrease; the oracle gives the smaller returned
residual and full contact the smallest linear residual.
A core/ordinary-endpoint residual-minimizing combination gives $0.322594$
on the same reference, against $0.002815$ for the positive spectral family.
These are different direction classes whose nonlinear ordering depends on the problem.

The selected construction needs $128$ restricted responses and two core
solves, avoiding both the full compliance and the full response cache.
Raw/restricted setup costs $2.740$~s after the root is available; selection
costs $0.243$~s and $131$ sparse contact actions.
Its $2.327$~MB response data compare with the earlier $886.346$~MB full
response cache. Scalar core solves still transfer full vectors through CPU memory.
The full reference inherits $32.510$~s and $2272$ compliance RHSs, plus
$0.643$~s fresh loading/reduced setup.
These setup timings come from different runs and do not establish a
matched comparison of total time from initialization.
The complete histories below measure its effect on total work.

\FloatBarrier
\subsection{Complete trajectories and their cost}
\label{sec:trajectory-summary}

The smaller $150$-Pa finger/cup history is frictionless, with one initial
material-core factor and forcing $0.1$.
Larger histories refresh the complete noncontact core $\overline M$ at every
equation, use restart $80$, a $480$-column linear limit and a $128$-iteration
nonlinear limit. Initial states, controls and nonlinear stopping criteria are shared.

\begin{table}[H]
\centering\small
\caption{Complete nonlinear comparisons. Solve columns include every
preliminary selected direction and subsequent Krylov column, with setup
responses counted separately. Times include preparation and first use.
The larger core-only $10^{-5}$ entry includes its terminal linear-work
observation and yields no time to solution.}
\label{tab:trajectory-summary}
\begin{tabular}{@{}llrrrrl@{}}
\toprule
Problem & Method & $\eta$ & Newton & Columns & Time (s) & Outcome\\
\midrule
Smaller & Direct & --- & 11 & --- & 252.919 & Converged\\
& Core only & $0.1$ & 12 & 22 & 243.667 & Converged\\
& Core + raw 8 & $0.1$ & 11 & 13 & 241.301 & Converged\\
\midrule
Larger & Direct & --- & 78 & --- & 901.397 & Converged\\
& Current core & $10^{-5}$ & 9 & 1057 & 714.833 & Iteration limit\\
& Current core & $0.1$ & 120 & 5231 & 1331.118 & Converged\\
& Ordinary raw 128 & $10^{-5}$ & 112 & 1831 & 1318.959 & Converged\\
& Ordinary raw 128 & $0.1$ & 128 & 1614 & 1778.749 & Iteration limit\\
& Selected strength, 128 & $0.1$ & 114 & 122 & 1886.882 & Converged\\
& Full supported contact & $0.1$ & 67 & 67 & 4032.666 & Converged\\
\bottomrule
\end{tabular}
\end{table}

On the smaller trajectory, the two iterative methods share the same
mechanical-reference construction and complete-residual checks.
Rank-eight correction reduces Krylov work; these single-run timings
do not establish an additional runtime benefit over the matched
core-only method.

Selected strength returns $110$ of $114$ preliminary directions directly;
the other four equations add two GPU columns each.
Its $14431$ setup responses cost $365.969$~s, and complete first-direction
work costs $548.807$~s. GPU continuation caches are included.
Core renewal costs $168.920$~s after $104.394$~s initial construction.

Full contact acquires $122940$ responses for $2939.896$~s.
Its $67$ directions use $134$ core applications and already meet the residual
criterion; the largest reduced arrays occupy $84.870$~MB.
Core-only uses $5231$ GPU applications for $108.485$~s and renewals for
$178.593$~s. These acquisition costs explain why fewer Krylov columns need not
reduce the total time.

Converged states are valid with consistent friction models.
Core-only, selected, full and tighter ordinary final absolute residuals are
$4.176\times10^{-7}$, $1.477\times10^{-7}$, $3.456\times10^{-7}$ and
$1.410\times10^{-7}$.
Selected strength has lower total time than the explicit full-contact construction;
tighter ordinary raw is the cheaper converged corrected alternative.

\subsection{Endpoint fields and friction history}
\label{sec:history-consequence}

\begin{table}[H]
\centering\small
\caption{Relative endpoint differences from the direct reference on the
larger problem. Force entries compare the concatenated normal and friction reaction
vectors over all ordered body pairs, using the Euclidean norm.
They compare coupled histories and have no automatic interpretation as
errors against a unique equilibrium.}
\label{tab:trajectory-fields}
\begin{tabular}{@{}lrrr@{}}
\toprule
Method & Displacement & Normal forces & Friction forces\\
\midrule
Core only, $\eta=0.1$ & $2.5112\times10^{-4}$ & $5.5948\times10^{-3}$ & $3.2616\times10^{-1}$\\
Selected strength, $\eta=0.1$ & $4.9908\times10^{-5}$ & $1.2671\times10^{-3}$ & $7.2205\times10^{-2}$\\
Ordinary raw 128, $\eta=10^{-5}$ & $1.8070\times10^{-5}$ & $2.2356\times10^{-5}$ & $8.3550\times10^{-5}$\\
Full supported contact & $1.3351\times10^{-6}$ & $1.2629\times10^{-5}$ & $2.6773\times10^{-5}$\\
\bottomrule
\end{tabular}
\end{table}

The norm of the direct friction-force array is $6.8286\times10^{-4}$ in the stored force units.
Core-only and selected strength differ by $2.2272\times10^{-4}$ and
$4.9306\times10^{-5}$ in absolute vector norm, with maximum coordinate
differences $0.008029$ and $0.001400$~mm.
Initial histories, final contact-pair identities and interpolation maps agree;
some per-pair history references differ.

This distinction enters the force law explicitly. Let
$T\in\mathbb R^{3\times2}$ be an orthonormal tangent basis,
$w_i$ the four point--triangle lifting weights, $z_i$ the current contact
vertices, and $z_i^-$ their history references. At an interior
closest-point configuration, $T^T\sum_iw_i z_i=0$. The tangential
increment is therefore
\[
 d_t=T^T\sum_iw_i(z_i-z_i^-)=-T^T\sum_iw_i z_i^-.
\]
Let $f_n>0$ be the normal-force magnitude. In the sliding regime
$\norm{d_t}_2\ge\varepsilon_f$, the point contribution to the friction
gradient is
\[
 \boldsymbol g_t=\mu_f f_n T\,\frac{d_t}{\norm{d_t}_2}.
\]
All observed final pairs are in this regime, with
$\norm{d_t}_2/\varepsilon_f$ between $6.87$ and $49.37$.
After applying the body orientation and reaction sign, this fixed-state
expression reproduces the recorded force vectors to relative discrepancies at most $1.54\times10^{-14}$.

Exchanging only history references in that frozen force map gives a
two-order symmetric decomposition. The history contribution has norm
$3.5608\times10^{-5}$, compared with total selected/direct difference
$3.4865\times10^{-5}$ for four independent cup-force vectors.
The remaining coefficient contribution is $9.6682\times10^{-7}$;
these vectors partly cancel. Tighter ordinary and full-contact histories
retain reference coordinates within about $10^{-11}$ of direct.
This static decomposition supplies no unique percentage attribution to
the nonlinear path.

Per-pair references are established at contact entry along accepted
segments. They remain inputs to the friction evaluator, so a strict final
residual does not identify one common history. The endpoint differences
qualify the physical interpretation of the inexact methods.

\FloatBarrier
\subsection{Response renewal, reset and precision diagnostics}
\label{app:reset-diagnostics}

For already assembled CPU matrices, constructing the material/complete
factors costs $19.469/19.082$~s, retained acquisition $0.307/0.067$~s, and full
supported acquisition $1.503$~s.
These are eight-column setup observations, distinct from the scalar
complete-workload preparation.
At the two changed cores, one/two-solve residual corrections and full
eight-column renewal reduce old-response defects $0.01937/0.04198$;
renewal reaches about $2\times10^{-11}$.
All four spaces still take five/six columns with their actual current images.
The response-space comparison has a precise quadratic-form definition.
Let $V$ contain the old responses and $V_{\rm new}$ the renewed ones.
For the current material--volumetric core $M_c$, set
$\widehat M_c=(M_c+M_c^T)/2$ and
\[
 G=V^T\widehat M_cV,\qquad
 T=G^{-1}V^T\widehat M_cV_{\rm new},\qquad \Delta V=V_{\rm new}-VT.
\]
The recorded statistic
\[
 \varepsilon_{\rm span}
 =\sqrt{\frac{\operatorname{tr}(\Delta V^T\widehat M_c\Delta V)}
                   {\operatorname{tr}(V_{\rm new}^T\widehat M_cV_{\rm new})}}
\]
is $0.006983$ and $0.015730$ for the two cores; $\operatorname{tr}$ denotes
the matrix trace.
It measures the renewed responses left outside the old span by this
projection. The restricted $G$ is numerically positive definite, and the
evaluated quotient has a positive denominator and nonnegative numerator.
Only these restricted quadratic forms are checked.
The nonzero values distinguish a changed response space from an invertible
change of basis.
Computing the residual and its singular vectors costs $0.277/0.198$~s before correction,
exceeding the observed acquisition saving.

The complete short-study model/solver preparation costs $71.4$--$72.6$~s.
Direct retains one symbolic analysis and performs $59$ numerical factors;
each iterative method has one factor.
Compressed acquisition uses eight responses and one contact root, with
$58$ later reuses of Krylov working memory. Full caching acquires $60$ responses
after the first nonzero contact difference.
Stable-support compressed updates upload $4416$ bytes per equation.

\paragraph{Short-study limitations.}
The reverse material-reference run takes an extra iteration at $140$~Pa.
Its eleventh Armijo margin changes from $+4.194\times10^{-14}$ to
$-5.634\times10^{-14}$ at energy about $0.02527$, leading to a half-sized
trial. Both runs meet the nonlinear stopping criteria, with relative displacement/normal-reaction
differences $1.432\times10^{-6}/5.061\times10^{-4}$ and maximum coordinate
difference $6.736\times10^{-7}$~mm.
Later displacement differences fall to $1.052\times10^{-10}$ and
$3.462\times10^{-13}$.
Contact-identity checks agree; the separate effects of the candidate and
arithmetic differences remain unresolved.

All methods stop at the $130$-to-$160$~Pa reversal while making a new contact
reference. The clearance-onset query exhausts $4096$ refinements about
$1.22\times10^{-9}$~mm outside the $0.01$~mm activation surface, with distance
tolerance $10^{-12}$~mm. This is a geometric-query limitation; the returned
state has not converged.

\paragraph{Extended-study limitations.}
The largest forward compressed/direct relative displacement difference is
$1.770\times10^{-6}$, with maximum coordinate difference
$6.561\times10^{-7}$~mm and relative normal-reaction difference
$2.322\times10^{-5}$.
Direct repeats exhibit essentially the same variation.
All seven compared endpoints use six Newton iterations, with final fractions
$0.225$, $0.45$ or $0.9$ and absolute residuals about
$9.680\times10^{-7}$, $6.870\times10^{-7}$ or $1.249\times10^{-7}$.
Interpolation between direct endpoints by the recorded fractions predicts
the others within $1.790\times10^{-12}$~mm per coordinate.
The short-study accuracy bound consequently does not extend to this family.

The exceptional run follows a valid linear direction.
A separate passive diagnostic converges for two targets; its next candidate lowers
the residual from $1.220\times10^{-6}$ to $1.220\times10^{-7}$.
A $60$-digit evaluation of that captured state's $66{,}315$ cells gives cup
material-energy change $-5.812\times10^{-14}$, against recorded
$+1.647\times10^{-13}$; gradient formation contributes negligibly.
This identifies constitutive cancellation, whose stable evaluation has
established treatments \citep{shakeri2024stable}.
The original failed-trial vectors were not saved: this is a consequential
separate observation, not an exact reconstruction of that failure.
The production formula and all reported timings remain unchanged.

Other computations were observed on the shared machine during these
comparisons. The reported ranges and ratios remain descriptive observations
under that shared-machine uncertainty.

\section{An established spectral-strength control}
\label{sec:spectral-strength}

Selecting a mode and choosing its correction strength are different
decisions. Following the exact spectral preconditioner of
\citet[Eq.~(5.1)]{frangella2023nystrom}, we use the first omitted mode as a
reference strength. The construction below is that established formula
after scaling by the mechanical core.

Assume $M\succ0$. Let $v_1,\ldots,v_n$ be orthonormal eigenvectors of
$M^{-1/2}U^TUM^{-1/2}$, with eigenvalues
$\lambda_1\ge\cdots\ge\lambda_n\ge0$. Choose $1\le r<n$ with
$\lambda_r>0$, and set $\theta=\lambda_{r+1}$. The metric root has rows
$\sqrt{\lambda_i}v_i^TM^{1/2}$ for $i\le r$. Define the row scales and
the scaled correction by
\begin{equation}
 d_i=\sqrt{\frac{\lambda_i-\theta}{\lambda_i(1+\theta)}},\qquad
 D_\theta=\diag(d_1,\ldots,d_r),\qquad
 R_\theta=D_\theta R_r,\qquad P_\theta=M+R_\theta^TR_\theta.
\label{eq:tail-normalized-contact}
\end{equation}
Since $0\le d_i\le1$, the correction obeys
$0\preceq R_\theta^TR_\theta\preceq U^TU$ and $P_\theta\succ0$.
If a retained eigenvalue equals $\theta$, its scaled row vanishes.

Congruence by $M^{-1/2}$ diagonalizes the reference pair in the basis
$v_i$. For $i\le r$, the preconditioner eigenvalue is
$1+(\lambda_i-\theta)/(1+\theta)=(1+\lambda_i)/(1+\theta)$; for
$i>r$ it is one. The generalized eigenvalues of $(H,P_\theta)$ are
therefore $1+\theta$ on the retained space and $1+\lambda_i$ on its
complement. In particular,
\begin{equation}
 \kappa_2(P_\theta^{-1/2}HP_\theta^{-1/2})
 =\frac{1+\lambda_{r+1}}{1+\lambda_n}.
\label{eq:tail-normalized-spectrum}
\end{equation}
When the contact model has a null space, this condition number equals the
ordinary retained-mode endpoint, although its eigenvalue distribution differs.

In Remark~\ref{rem:partial-cluster}, the scaled preconditioner is
$\diag(t/s,1,1)$ and maps the RHS under $HP_\theta^{-1}$ to $s b$.
One Krylov column consequently gives the exact solution, while the ordinary
correction has the nonzero first residual derived there. The weighted
factor-space map in \eqref{eq:tail-normalized-contact} is a positive
semidefinite contraction; the earlier omitted-energy optimality statement
compares orthogonal projectors. These are different admissible classes.
Neither reference spectrum predicts the exact FGMRES history on the full
$A=H+N$.

The scaled action reuses the ordinary responses $Z_r$. Define
$S_\theta=I_r+D_\theta(R_rZ_r)D_\theta$. For an input vector
$v\in\R^n$, put $q=M^{-1}v$. Substitution in the Woodbury identity gives
\[
 P_\theta^{-1}v
 =q-Z_rD_\theta S_\theta^{-1}D_\theta R_rq.
\]
Thus changing these row strengths requires no additional core-response
columns once the metric space and its eigenvalues are available. The cost
of obtaining that space and its first omitted eigenvalue remains explicit.

\subsection{Normalization using only retained responses}
\label{app:retained-gram-strength}

A cheaper construction applies the same spectral recipe to the retained
surrogate. Assume $M\succ0$ and let $R\in\R^{r\times n}$ have independent
rows. Define $Z=M^{-1}R^T$ and the retained response Gram matrix
$G_R=RZ$. Write $G_R=W\operatorname{diag}(\mu_i)W^T$, where $W$ is
orthogonal and $\mu_1\ge\cdots\ge\mu_r>0$. Set $\theta_R=\mu_r$ and
\[
 D_R=\operatorname{diag}\left(
 \sqrt{\frac{\mu_i-\theta_R}{\mu_i(1+\theta_R)}}\right),
 \qquad L=D_RW^T,\qquad P_R=M+R^TL^TLR.
\]
This is the normalization in \citet[Eq.~(1.3)]{frangella2023nystrom}
applied to the retained surrogate; a raw contact projection has its own
approximation properties. The mode associated with $\mu_r$ receives zero
correction weight. Woodbury gives
\[
 P_R^{-1}=M^{-1}-ZL^T(I_r+LG_RL^T)^{-1}LRM^{-1}.
\]
The $r$ existing responses determine this action without a full contact
compliance. The numerical experiment retains the measured small product
in the reduced solve, with its asymmetry reported separately.

For the crossed diagnostic we also admit a prescribed $\theta\ge0$ and
replace the squared row weight by
$(\mu_i-\theta)_+/[\mu_i(1+\theta)]$, where $(s)_+=\max\{s,0\}$.
This defines the positive-part control when the prescribed value exceeds
some retained eigenvalues. Let $D_\theta$ have the square roots of these
weights on its diagonal and set $L_\theta=D_\theta W^T$. The weights remain
between zero and one. For every nonzero vector $x$,
$x^TP_\theta x=x^TMx+\norm{L_\theta Rx}_2^2>0$, where
$P_\theta=M+R^TL_\theta^TL_\theta R$. The same Woodbury identity applies. Acquisition of a
prescribed diagnostic strength is a separate cost.

The retained data provide no uniform bound on an unseen contact spectrum.
For a five-dimensional example, let $\tau\ge1$ and define
\[
 C=\operatorname{diag}(5,4,3,2,1),\quad U=C^{1/2},\quad
 M_\tau=\operatorname{diag}(1/20,4,3/\tau,2/\tau,1/\tau).
\]
Let $e_i$ be the coordinate vectors and set $Q=(e_1,e_2)$ and $R=Q^TU$.
The raw selector takes this space because it contains the two largest
contact eigenvalues. Direct multiplication yields
\[
 G_\tau=UM_\tau^{-1}U^T=\operatorname{diag}(100,1,\tau,\tau,\tau),
 \qquad G_R=\operatorname{diag}(100,1).
\]
Both $G_R$ and the complete retained responses $M_\tau^{-1}R^T$ are
independent of $\tau$. The normalization therefore uses $\theta_R=1$,
while the third eigenvalue of $G_\tau$ is $\tau$. The first diagonal of
$P_R$ is $101/40$; the corresponding diagonal of $M_\tau+C$ is $101/20$.
Dividing all five diagonal entries gives generalized eigenvalues
$2,2,1+\tau,1+\tau,1+\tau$ and condition number $(1+\tau)/2$.
Moreover, $G_\tau Q=QG_R$, so the Ritz residual outside $Q$ is zero and
repeated $G_\tau$ actions preserve that same space. Here $M_\tau$ and $C$ commute.

This example limits conclusions drawn from the retained Gram or its Ritz
residual alone. Information reaching the complement, or an independently justified direction-specific criterion using the complete operator, is
needed for stronger conclusions. It gives no unique causal explanation
of a nonlinear contact trial.

\subsection{Determining strength from the current right-hand side}
\label{app:rhs-strength}

Residual-minimizing preconditioner combinations are established in
multi-preconditioned GMRES \citep{greif2011multipreconditioning}.
\citet{ayuso2014combined} derive a two-parameter minimization and share an
existing preconditioner application. Variable-weight Schwarz also exploits
shared subdomain solves and sparse actions \citep{greif2013weights}.
Here those principles are specialized to the positive spectral family above.
The contact support permits selection without acquiring the complete contact
compliance or the omitted spectrum.

Let $E\in\R^{p\times n}$ select $p$ distinct contact coordinates, so that
$EE^T=I_p$. Assume $M\succ0$ and write $C=E^TC_sE\succeq0$ and $H=M+C$.
Let $R=R_sE$ have $r\ge1$ rows, with arbitrary rank, where $R_s\in\R^{r\times p}$.
Define $Z=M^{-1}R^T$, $Z_s=EZ$ and
$G_R=R_sZ_s=W\diag(\mu_1,\ldots,\mu_r)W^T$, with $W$ orthogonal and
$\mu_i\ge0$. For $\theta\ge0$, define
\begin{equation}
 d_i^2(\theta)=\begin{cases}
 \dfrac{(\mu_i-\theta)_+}{\mu_i(1+\theta)},&\mu_i>0,\\
 0,&\mu_i=0,
 \end{cases}
 \qquad P_\theta=M+R^TW\diag(d_i^2(\theta))W^TR.
 \label{eq:rhs-positive-family}
\end{equation}
Each member is positive definite. The endpoints $\theta=0$ and
$\theta\ge\max_i\mu_i$ give $M+R^TR$ and $M$, respectively.

Fix a nonzero right-hand side $b\in\R^n$. Put $y=M^{-1}b$, $y_s=Ey$
and $b_s=Eb$. Let $b_o$ contain the coordinates of $b$ outside the support.
Define
\begin{align}
 \gamma_i(\theta)&=\begin{cases}
 \dfrac{(\mu_i-\theta)_+}{\mu_i(1+\mu_i)},&\mu_i>0,\\
 0,&\mu_i=0,
 \end{cases}\qquad
 c_\theta=W\diag(\gamma_i(\theta))W^TR_sy_s,\label{eq:rhs-coefficients}\\
 p_\theta&=y-Zc_\theta,\qquad
 p_s(\theta)=y_s-Z_sc_\theta.\label{eq:rhs-supported-direction}
\end{align}
The encoded right-hand side and operator image are the vectors in
$\R^{p+1}$ given by
\begin{equation}
 \bar b=\begin{pmatrix}\norm{b_o}_2\\b_s\end{pmatrix},\qquad
 \bar a(\theta)=\begin{pmatrix}\norm{b_o}_2\\
 b_s+C_sp_s(\theta)-R_s^Tc_\theta\end{pmatrix}.
 \label{eq:rhs-support-encoding}
\end{equation}
Write the complete contact root as $U=U_sE$, defining $U_s$ on the support.
For a vector $v\in\R^p$, its sparse action is $C_sv=U_s^T(U_sv)$.

\begin{proposition}[Supported first-direction selection]
\label{prop:rhs-strength-selection}
Under the preceding assumptions, $p_\theta=P_\theta^{-1}b$ and, for every
real $\alpha$,
\begin{equation}
 \norm{b-\alpha Hp_\theta}_2
 =\norm{\bar b-\alpha\bar a(\theta)}_2.
 \label{eq:rhs-encoded-residual}
\end{equation}
Let $0=\tau_0<\tau_1<\cdots<\tau_k$ be the sorted distinct values in
$\{0,\mu_1,\ldots,\mu_r\}$. For $k>0$, the global minimum over
$\theta\ge0$ and $\alpha\ge0$ is obtained by comparing the $k$
two-variable nonnegative least-squares problems
\begin{equation}
 \min_{t_0,t_1\ge0}
 \norm{\bar b-t_0\bar a(\tau_j)-t_1\bar a(\tau_{j+1})}_2,
 \qquad 0\le j<k.
 \label{eq:rhs-endpoint-cone}
\end{equation}
For a minimizing pair with $t_0+t_1>0$, the corresponding parameters are
$\alpha=t_0+t_1$ and
$\theta=(t_0\tau_j+t_1\tau_{j+1})/(t_0+t_1)$.
For $k=0$, use the single core-only ray $\{\alpha M^{-1}b:\alpha\ge0\}$.
\end{proposition}

\begin{proof}
Let $w_i$ be column $i$ of $W$. If $\mu_i=0$, then
$\norm{M^{-1/2}R^Tw_i}_2^2=w_i^TG_Rw_i=0$, so $R^Tw_i=0$.
The zero-mode weights therefore change neither the correction nor its
inverse. In particular the zero-strength endpoint is $M+R^TR$ even when
$R$ has dependent rows. If every eigenvalue is zero, $R=0$ and the family
is core-only. For the positive modes, Woodbury applied to \eqref{eq:rhs-positive-family} diagonalizes its small
inverse in the basis $W$. For an active eigenvalue $\mu_i>\theta$, the
inverse-update coefficient is
\[
 \frac{(\mu_i-\theta)/[\mu_i(1+\theta)]}
 {1+(\mu_i-\theta)/(1+\theta)}
 =\frac{\mu_i-\theta}{\mu_i(1+\mu_i)}.
\]
An inactive eigenvalue contributes zero. This proves
$P_\theta^{-1}b=y-Zc_\theta$. Since $My=b$ and $MZ=R^T$,
\[
 Hp_\theta=b+E^T\bigl(C_sp_s(\theta)-R_s^Tc_\theta\bigr).
\]
Outside the support the residual is $(1-\alpha)b_o$. Squaring its norm
and adding the supported residual proves \eqref{eq:rhs-encoded-residual}.

Within $[\tau_j,\tau_{j+1}]$, every coefficient $\gamma_i$ is affine in
$\theta$. Thus $p_s$ and $\bar a$ interpolate their endpoint values.
For $\theta=(1-s)\tau_j+s\tau_{j+1}$, with $0\le s\le1$, set
$t_0=\alpha(1-s)$ and $t_1=\alpha s$. These coefficients are nonnegative
and give \eqref{eq:rhs-endpoint-cone}. Conversely, any feasible pair with
positive sum gives the displayed $\alpha$ and $\theta$; the zero pair
represents $\alpha=0$. For $\theta\ge\tau_k$ the direction is constant
and is already represented by the last endpoint ray. Comparing all interval
minima therefore solves the declared global problem.
\end{proof}

Each two-variable problem can be solved by comparing its feasible
unconstrained least-squares solution, its two endpoint rays and zero.
Rank-deficient endpoint images are permitted. The feasible set is a union
of adjacent nonnegative endpoint cones. It lies inside the unrestricted
span of all endpoint images used by a first-step multi-preconditioned
projection. Consequently this construction does not claim a better residual
against that larger projection space. Its restriction preserves the
specified positive spectral family. When the minimizing $\alpha$ is positive,
$b^T(\alpha p_\theta)>0$ follows from $P_\theta\succ0$. If $-b$ is the
potential gradient, this is initial potential descent. Unit-step acceptance,
globalization cost and nonlinear convergence are distinct questions.

Only $Z_s$, $R_s$ and a contact action are required for selection after the
first core solve. The full selected direction has the implicit realization
\[
 p_\theta=M^{-1}\bigl(b-E^TR_s^Tc_\theta\bigr).
\]
Thus $r$ restricted setup responses and two core solves determine the first
direction without storing $Z\in\R^{n\times r}$. Selection involves
supported arithmetic for every spectral interval; this cost remains part of
the algorithm. Repeated eigenvalues reduce the number of intervals. The
chosen parameter may be held fixed for later Krylov columns, though the
first-direction proposition does not yield a multi-column convergence rate.

For computed responses, the predicted and actual actions require a separate
comparison. Let $p$ be a computed implicit direction, $c$ its retained
coefficient vector, and $p_s^{\rm pred}$ its computed supported prediction.
Define
\[
 a^{\rm pred}=b+E^T(C_sp_s^{\rm pred}-R_s^Tc),\qquad
 q=b-R^Tc-Mp.
\]
For the original matrix $A$, direct subtraction gives the exact identity
\begin{equation}
 Ap-a^{\rm pred}=(A-H)p-q+E^TC_s(Ep-p_s^{\rm pred}).
 \label{eq:rhs-differential-errors}
\end{equation}
Indeed, $Mp=b-R^Tc-q$ and $Hp=Mp+E^TC_sEp$.
Writing the right-hand side of \eqref{eq:rhs-differential-errors} as
$\delta$, the reverse triangle inequality gives
\[
 \left|\norm{b-\alpha Ap}_2-\norm{b-\alpha a^{\rm pred}}_2\right|
 \le |\alpha|\norm{\delta}_2.
\]
This accounts separately for the operator split, implicit core solve and
supported-response error. It also retains their vector cancellation.
The supported optimum alone certifies no property of a computed
original-$A$ action.

Let $w$ be a reported unit eigenvector associated with a nonpositive
eigenvalue estimate. The numerical implementation replaces that estimate
by the factored Rayleigh form $(R_s^Tw)^T(Z_sw)$, evaluated in extended
precision. This expression equals $w^T(R_sZ_s)w$ in exact arithmetic and
uses the stored factors directly. Exact zeros use the kernel formulas above;
a negative evaluated Rayleigh form remains a failure of the construction. The same
working spectrum is passed to any subsequent root normalization. No
numerical-rank cutoff or negative-mode clipping is used. The original-$A$
action check remains necessary for these approximate spectral data.

\section{Reproducibility information}

The device experiments use cuDSS 0.8.0, JAX 0.9.2, NumPy 2.4.4 and
SciPy 1.17.1, with float64 arithmetic on a Tesla V100S-PCIE with 32~GB
of memory. The physical inputs are specified in \cref{sec:protocol}.
Each process is restricted to one CPU core. Matched GPU warm comparisons
and CPU-controlled fixed-equation comparisons retain their distinct
timing scopes. Complete histories use the stated nonlinear criteria and
report incomplete iteration limit outcomes alongside converged results.

For the primary frozen case, the barrier scale follows a GIPC-style
numerical build-time rule \citep{huang2024gipc}, evaluated in the stated
millimetre coordinate convention. Let $\overline m_{\rm cal}>0$ be its
calibration scalar: the mean of supplied positive nodal masses when available,
and one otherwise. This scalar does not add physical mass to a static equation.
Let $L_{\mathrm{bbox}}$ be the mesh bounding-box diagonal, and define the
implemented logarithmic-barrier derivative expression, for $a,b>0$, by
\[
 H_b(a,b)=-2\log(a/b)-4(a-b)/a+(a-b)^2/a^2.
\]
Let $\widehat d$ be the contact-distance argument supplied to that routine and set
$d_{\mathrm{ref}}=10^{-16}L_{\mathrm{bbox}}^2$.
The barrier coefficient is
\[
 \kappa_{\mathrm{ipc}}
 =\left|\frac{10^{11}\overline m_{\rm cal}}
 {4d_{\mathrm{ref}}H_b(d_{\mathrm{ref}},\widehat d)}\right|.
\]
It is $466.00868$ for the primary extracted case. Frozen operators retain
their actual numerical contact tangents. This calibration is separate from
the nonlinear friction-history interpretation.

\section{Dimension and rank conventions}

The contact factor is row oriented: $U\in\mathbb R^{m\times n}$ and the
positive-semidefinite update is $U^TU$. The primary Tet10/P0--IPC core
is $K_0+B^TD^{-1}B$; the large complete trajectories instead use the
current complete noncontact core $\overline M$.
A support restriction has $p$ rows, and the retained factor-space
dimension is denoted by $r$. These dimensions play different roles.

For the primary extraction, $U\in\mathbb R^{94\times325260}$ touches
$60$ free coordinates and has structural rank $58$. Let $\sigma_i$ be
its singular values in descending order. The reported numerical-rank
diagnostic uses
\begin{equation}
 \tau_{\mathrm{rank}}=\max\{10^{-12},10^{-10}\sigma_1\},
 \qquad \sigma_i>\tau_{\mathrm{rank}}.
 \label{eq:factor-rank-threshold}
\end{equation}
Here $\sigma_1=52.63054$, $\sigma_{42}=2.1212\times10^{-5}$ and
$\sigma_{43}=4.3412\times10^{-11}$, giving numerical rank $42$.
This diagnostic leaves the experimental retained dimensions unchanged:
$r=94$ denotes the full stored-row endpoint. Contact-pair counts,
supported coordinates and numerical rank are reported separately.

\begingroup
\raggedright

\endgroup
\end{document}